\pdfoutput=1
\documentclass{article}

\usepackage[utf8]{inputenc}
\usepackage[T1]{fontenc}
\usepackage{arxiv}
\usepackage{comment}
\usepackage{url}
\usepackage{microtype}

\usepackage{amsmath}
\usepackage{amssymb}
\usepackage{mathrsfs}
\usepackage{amsthm}
\usepackage{amsfonts}
\usepackage{mathtools}
\usepackage{stmaryrd}
\usepackage{bm}
\usepackage{bbold}

\allowdisplaybreaks
\DeclarePairedDelimiter\abs{\lvert}{\rvert}
\DeclarePairedDelimiter\norm{\lVert}{\rVert}

\usepackage{booktabs}
\usepackage{multirow}
\usepackage{nicefrac}
\usepackage{multicol}
\usepackage{diagbox}
\usepackage{array}
\usepackage{enumerate}
\usepackage[normalem]{ulem}

\usepackage{graphicx}
\graphicspath{ {./figs/} }
\usepackage{caption}
\usepackage{subcaption}

\usepackage{tikz}
\usetikzlibrary{shapes.geometric, arrows.meta, positioning}

\definecolor{myyellow}{RGB}{235, 175, 0}
\definecolor{myred}{RGB}{230, 30, 30}
\definecolor{mycyan}{RGB}{0, 180, 240}
\definecolor{mypurple}{RGB}{120, 30, 160}
\definecolor{myblue}{RGB}{20, 50, 90}

\usepackage[ruled,vlined]{algorithm2e}
\usepackage[toc,header]{appendix}
\usepackage[tickmarkheight=0.1cm]{todonotes}

\numberwithin{equation}{section}
\newtheorem{theorem}{Theorem}[section]

\newtheorem{corollary}{Corollary}[section]

\newtheoremstyle{iremark}
  {\topsep}{\topsep}{\upshape}{0pt}{\itshape}{.}{5pt plus 1pt minus 1pt}
  {\thmname{#1}\thmnumber{ \itshape#2}\thmnote{ (#3)}}

\theoremstyle{iremark}
\newtheorem{remark}{Remark}

\newenvironment{nalign}{\begin{equation}\begin{aligned}}{\end{aligned}\end{equation}\ignorespacesafterend}

\newcommand{\spc}{{\hspace{0.5cm}}}

\newcommand{\jmh}{{j-\frac{1}{2}}}
\newcommand{\jph}{{j+\frac{1}{2}}}
\newcommand{\half}{\frac{1}{2}}

\newcommand*\dif{\mathop{}\!\mathrm{d}}

\newcommand{\iph}{i+\frac{1}{2}}
\newcommand{\imh}{i-\frac{1}{2}}

\newenvironment{frcseries}{\fontfamily{frc}\selectfont}{}

\newcommand{\no}{\nonumber}

\newcounter{example}
\renewcommand{\theexample}{\arabic{example}}
\newenvironment{example}[1][]{%
    \refstepcounter{example}%
    \noindent\textbf{Example \theexample.} %
    \ifx&#1&\else\label{#1}\fi
}{\par}

\usepackage{hyperref}
\hypersetup{
    colorlinks=true,
    linkcolor=blue,
    citecolor=magenta,
    urlcolor=red,
    linktoc=all,
    pdfborder={0 0 1}
}

\title{Semi-discrete Active Flux method for two-dimensional Hyperbolic systems with Coriolis Source Terms}

\author{
Nikhil Manoj\\
Institute of Mathematics\\
University of Wuerzburg\\
Wuerzburg, Germany, 97074\\
\texttt{nikhil.manoj@uni-wuerzburg.de} \\
\And
Wasilij Barsukow\\
Institut de Mathématiques de Bordeaux (IMB)\\
CNRS UMR 5251\\
351 Cours de la Libération\\
33405 Talence, France\\
\texttt{wasilij.barsukow@math.u-bordeaux.fr}\\
\emph{Present address:}
Imperial College London\\
CNRS IRL 2004\\
Huxley Building\\
South Kensington Campus\\
London, United Kindgom\\
\And
Christian Klingenberg\\
Institute of Mathematics\\
University of Wuerzburg\\
Wuerzburg, Germany, 97074\\
\texttt{christian.klingenberg@uni-wuerzburg.de} \\
}

\begin{document}

\maketitle

\begin{abstract}
We investigate the semi-discrete Active Flux numerical method on Cartesian grids applied to two-dimensional linear and nonlinear shallow water equations with Coriolis source terms. These hyperbolic systems admit non-trivial stationary solutions governed by a geostrophic equilibrium. In this setting, we analyze the stationarity-preserving properties of our semi-discrete Active Flux formulation. To do so for the linear system, we apply a discrete spatial Fourier transform to the semi-discrete method and analyze the resulting evolution matrix. We show that its discrete kernel is a non-trivial discretization of the geostrophic equilibrium, which implies that the method preserves a discrete geostrophic equilibrium state. This is an inherent feature of Active Flux and is achieved without any additional modifications. Numerical experiments confirm the method's ability to maintain geostrophic stationary states for the linear system. Additionally, we numerically investigate the performance of the Active Flux method for the nonlinear shallow water system.
\end{abstract}





\section{Introduction}

A multitude of physical phenomena in the real world are modeled by multidimensional hyperbolic conservation laws with source terms. In particular, the shallow water equations are extensively utilized in oceanography, tsunami modeling, and meteorology. They constitute a nonlinear system of hyperbolic partial differential equations where source terms appear when modeling non-flat topography, Coriolis force effects due to the rotation of the Earth \cite{audusse2009, chertock2018, Liu2019, audusse2021, audusse2025, gonzalez2025geostrophic}, bottom friction, wind shear stress at the free surface etc. Of particular interest are the asymptotic regimes characterized by low Froude, low Rossby, and low Strouhal numbers \cite{audusse2018, chertock2018}. In certain asymptotic regimes, the nonlinear shallow water equations can be understood by suitable linearizations, balancing pressure gradients against rotational forces.  In this work, we are interested in one such linearization of the two-dimensional shallow water equations over a flat topography in a rotating frame. Specifically, we study the two-dimensional linear acoustics equations augmented with Coriolis source terms. Such linear hyperbolic systems with Coriolis source terms have been studied previously in \cite{bernard2008, thuburn2009, leroux2012,  audusse2018, audusse2021, audusse2025, barsukow2025b}, and exhibit non-trivial stationary states governed by a geostrophic equilibrium. An important property expected of numerical discretizations to such problems is well balancing, the ability to preserve non-trivial stationary states described by a balance of the flux and source terms, even on coarse meshes. Many physically relevant waves are perturbations of such steady states \cite{desveaux2022} and would otherwise become polluted by spurious waves originating in the method's failure to preserve the background equilibrium. Thus, numerical methods capable of preserving non-trivial stationary states are crucial for practical, long-time simulations \cite{castro2008, audusse2018, barsukow_mirco26}.

\par
In this article, we focus on the semi-discrete Active Flux numerical method \cite{abgrall_barsukow2023, Barsukow2025Fourier, barsukow_lisa_2026, duan2025}.  Active Flux is a generalized finite volume framework that employs a globally continuous spatial reconstruction, thus avoiding the use of  exact or approximate Riemann solvers. In addition to the cell averages which are evolved using a conservative formulation, Active Flux independently evolves point values located at the cell interfaces. Originally proposed by Roe \cite{EymannRoe2011ActiveFlux, EymannRoe2011ActiveFluxSystems, EymannRoe2013MultidimensionalActiveFlux}, the method has since been explored for various classes of hyperbolic equations \cite{barsukow2023, barsukow_2021_jsc, barsukow2021, barsukow_berberich24, duan_2025_MHD}. Two primary directions have emerged in recent years for implementing the time evolution in Active Flux: fully discrete and semi-discrete. The fully discrete Active Flux method \cite{barsukow_hohm_2019, barsukow_2021_jsc, roe2021} relies on deriving problem-specific exact or approximate evolution operators to update the interface point values in a single stage. The semi-discrete Active Flux strategy \cite{abgrall2023b, abgrall_barsukow2023,  barsukow_lisa_2026, duan2025} is based on a method-of-lines approach, combining a semi-discretization in space for both the point values and cell averages with a standard strong stability preserving Runge-Kutta (SSP-RK) \cite{gottlieb1998otalVD, gottliebsigalshu2001} time integration. The semi-discrete approach allows to apply the method to a wide variety of hyperbolic problems without significant changes in the algorithm.

\par The Active Flux method has been found to intrinsically  satisfy certain desirable properties, like the stationarity preservation property, in various scenarios, see \cite{barsukow_hohm_2019, barsukow2021, barsukow2023, Barsukow2025Fourier}.  For  \textcolor{black}{homogenoeus hyperbolic} linear systems in multiple spatial dimensions, \cite{barsukow2019} characterized the stationarity preservation of standard finite volume and finite difference methods. This idea was further explored for Active Flux in \cite{barsukow2023, Barsukow2025Fourier} and it was found that Active 
Flux naturally admits discretizations of the relevant stationary states without any modification. It is worth noting that this is not in general true for finite volume/ finite difference numerical methods, which require additional modifications to achieve stationarity preservation. 
Further, for hyperbolic systems with source terms, the Active Flux method has been studied in various scenarios. A linear problem was studied in \cite{barsukow2021}. For one-dimensional shallow water equations with topography, a well-balanced fully discrete Active Flux method was proposed in \cite{barsukow_berberich24}. The closely related PAMPA methods were proposed for one-dimensional shallow water equations with general source terms in \cite{abgrall2025pampaSWE}. For the ideal MHD equations appended with a Godunov-Powell source term, an Active Flux method was designed in \cite{duan_2025_MHD}, and a PAMPA method for ideal MHD was  studied in \cite{liu_2025MHD}.  

\par 
The main objective of this work is to propose and investigate the stationarity preservation properties of a semi-discrete Active Flux numerical method for two classes of hyperbolic systems with Coriolis source terms: (i) linear acoustics system, and (ii) shallow water system.  In the same spirit of  \cite{barsukow2019, barsukow2023, Barsukow2025Fourier}, for the linear acoustics system with Coriolis source terms, we use a Fourier transform analysis on the proposed method, to obtain a semi-discrete evolution matrix. The ability of the numerical method to preserve nontrivial stationary states can be understood by studying the existence of a nontrivial kernel for this evolution matrix, see \cite{barsukow2019}. Using this idea, we establish that the method admits a nontrivial discrete stationary state. In other words, we show that there exists a discretization of the analytical geostrophic equilibrium state that is preserved by the numerical method. Furthermore, a fully discrete von Neumann stability analysis reveals that the inclusion of the rotational source term imposes a stability threshold which is the same as was found for pure acoustics in \cite{Barsukow2025Fourier}. Additionally, we also investigate the Active Flux method for the nonlinear shallow water equations with Coriolis source terms \cite{chertock2018, audusse2025}. The method is demonstrated to perform well for the nonlinear system, especially in low Froude number regimes, without requiring any structural modifications.

\par
The remainder of this article is organized as follows.  Section \ref{sec:system} elaborates on the hyperbolic system of PDE under consideration and their stationary states. In Section \ref{sec:afscheme}, we present the Active Flux numerical method to approximate the PDE. The Fourier transform of the proposed numerical method is investigated in Section \ref{sec:fourier} and the kernel of the resultant semi-discrete evolution matrix is analyzed to understand its stationarity preservation property. Numerical experiments for the linear and nonlinear systems form the content of Section \ref{sec:numericalexp}. We conclude the article with some remarks in Section \ref{sec:conclusion}. Technical details pertaining to the Fourier analysis are detailed in the Appendix.

\section{Hyperbolic systems with Coriolis source terms}\label{sec:system}
Consider the dimensionless shallow water equations over a flat topography (see \cite{audusse2018, audusse2017, audusse2025}) with Coriolis source terms, given by
\begin{align}\label{eq:swe_noncons}
    \text{St} \partial_t h + \nabla \cdot (h {\boldsymbol{u}}) &= 0,  \\ \text{St} \partial_t (h {\boldsymbol{u}}) + \nabla \cdot (h {\boldsymbol{u}} \otimes {\boldsymbol{u}}) + \frac{1}{\text{Fr}^2} \nabla \left(\frac{h^2}{2}\right) &= -\frac{1}{\text{Ro}} h {\boldsymbol{u}}^\perp, \nonumber
\end{align}
where $h$ and ${\boldsymbol{u}} = ({u}, {v})$ denote the water depth and velocity, ${\boldsymbol{u}}^{\perp} = (-{v}, {u}).$ For constants $U, H, L,$ and $T$ representing the characteristic velocity, vertical length, horizontal length, and time scales of the flow,  $\text{St}$, $\text{Fr}$, and $\text{Ro}$ are the Strouhal, Froude, and Rossby numbers, respectively defined as $\displaystyle \mathrm{St} = \frac{L}{U T}, \, \mathrm{Fr} = \frac{U}{\sqrt{g H}}, \, \mathrm{Ro} = \frac{U}{\Omega L},$
where $g$ and $\Omega$ denote the gravitational acceleration and the angular velocity of the Earth, respectively. 

\subsection{Low Froude - low Rossby asymptotics}
For large-scale oceanographic applications, typical characteristic values are of the order of (see \cite{audusse2017})
\begin{equation}
    U \approx 1 \, \mathrm{m \cdot s^{-1}}, \quad L \approx 10^6 \, \mathrm{m}, \quad H \approx 10^3 \, \mathrm{m}, \quad \Omega \approx 10^{-4} \, \mathrm{rad \cdot s^{-1}}.
\end{equation}
Under these physical conditions, $\mathrm{Fr} = 10^{-2}$ and $\mathrm{Ro} = 10^{-2}$. We are thus led to consider the asymptotic behaviour with $\mathrm{Fr} = \mathcal{O}(\epsilon)$ and $\mathrm{Ro} = \mathcal{O}(\epsilon)$ in the limit of  $\epsilon \rightarrow 0$. 
Now, considering small Froude and Rossby numbers and an $\mathcal{O}(1)$ Strouhal number  {(or equivalently, consider long time evolution)}, i.e, 
$\mathrm{Fr} = \mathrm{Ro} = \epsilon, \, \mathrm{St} = \mathcal{O}(1),$  noting that $\displaystyle \nabla \left(\frac{h^2}{2}\right) = h \nabla h$ and rescaling the discharge equations in \eqref{eq:swe_noncons} yields
\begin{nalign}\label{eq:swe_noncons_rescaled}
    \partial_t h + \nabla \cdot (h {\boldsymbol{u}}) &= 0,  \\  \epsilon^2\left(\partial_t (h {\boldsymbol{u}}) + \nabla \cdot (h {\boldsymbol{u}} \otimes {\boldsymbol{u}}) \right) +   h \nabla h &= -\epsilon h {\boldsymbol{u}}^\perp, 
\end{nalign}

Consider the formal asymptotic expansions 
\begin{align}\label{eq:asy_exp}
&h = h_0 + \epsilon h_1 + \epsilon^2 h_2 + \mathcal{O}(\epsilon^3),
&{\boldsymbol{u}} = {\boldsymbol{u}}_0 + \epsilon {\boldsymbol{u}}_1 + \epsilon^2 {\boldsymbol{u}}_2 + \mathcal{O}(\epsilon^3),
\end{align}
which when substituted into the discharge equation in \eqref{eq:swe_noncons_rescaled},
in the leading order yields 
\begin{align}\label{eq:lakeatrest_init}
    h_0 \nabla h_0 &= 0.
\end{align} From \eqref{eq:lakeatrest_init}, assuming that $h_0 > 0,$ one obtains the \emph{lake-at-rest background state}
\begin{align}\label{eqa:lar}
    \nabla h_0 &= 0.
\end{align} At the next order, we obtain
\begin{align*}
    h_0 \nabla h_1 = - h_{0} \boldsymbol{u}_{0}^{\perp}.
\end{align*} Thus, we obtain the following \emph{geostrophic equilibrium}
\begin{align}\label{eq:geo_eq}
   \nabla h_1 = - \boldsymbol{u}_{0}^{\perp}.
\end{align}

\begin{remark}
    We note that the continuity equation in \eqref{eq:swe_noncons} does not contribute any further equations at this order that would be independent from the one already obtained. In fact, substituting the expansions \eqref{eq:asy_exp} into the continuity equation, together with the assumption of a lake-at-rest background state that satisfies \eqref{eqa:lar}, results in the divergence-free condition
    \begin{align}\label{geo_eq_divfree}
        \nabla \cdot \boldsymbol{u}_0 = 0.
    \end{align} 
    If a solution satisfies \eqref{eq:geo_eq}, then it straightforward that $ \nabla \cdot \boldsymbol{u}_0 = \nabla \cdot (- \nabla h_{1})^{\perp} = 0.$ One thus observes that the geostrophic equilibria are given by an under-determined system of equations (two equations \eqref{eq:geo_eq} for the three unknowns $\boldsymbol{u}_0, h_{1}$).
\end{remark}

\begin{remark}
    The asymptotic limit yielding the geostrophic equilibrium state \eqref{eq:geo_eq} need not be a stationary solution to the nonlinear system \eqref{eq:swe_noncons}. In other words, for an initial datum which satisfies \eqref{eq:geo_eq}, it is not true in general that
    \begin{align}
        \partial_t h = 0, \quad \partial_t (h \boldsymbol{u}) &= 0. 
    \end{align}
\end{remark}

\subsection{Linearization around a state of rest}\label{subsec:linearization}
Next, we linearize the system \eqref{eq:swe_noncons} around a state of state of rest with constant height and vanishing velocity \cite{thuburn2009}
\begin{align}
    h &= 1 + \delta \eta, \quad \boldsymbol{u} = \delta \boldsymbol{u},
\end{align}
with $\delta \ll 1$, neglecting $\mathcal{O}(\delta^2)$ terms and defining  $p := \eta / \mathrm{Fr},$ results in
\begin{equation*}
\begin{aligned}
    \mathrm{St}\partial_t p + \frac{1}{ \mathrm{Fr}} \nabla \cdot \boldsymbol{u} &= 0, \\
    \mathrm{St}\partial_t \boldsymbol{u} + \frac{1}{ \mathrm{Fr}} \nabla p &= -\frac{1}{ \mathrm{Ro}} \boldsymbol{u}^\perp.
\end{aligned}
\end{equation*}
Further, defining $\displaystyle a_* = \frac{1}{(\mathrm{St} \mathrm{Fr})}$ and $\displaystyle c = \frac{1}{(\mathrm{St} \mathrm{Ro})}$, we recover the standard linear acoustics equations with a Coriolis source term (see \cite{bernard2008, thuburn2009, leroux2012}) 

\begin{align}\label{eq:coriolis}
    \frac{\partial }{\partial t} {\boldsymbol{q}}+ \frac{\partial}{\partial x} \boldsymbol{f}^x(\boldsymbol{q}) + \frac{\partial}{\partial y} \boldsymbol{f}^y(\boldsymbol{q}) ={{\boldsymbol{s}(\boldsymbol{q})}},
    \end{align}
where $\boldsymbol{q} = \begin{pmatrix}
         u \\ v \\ p
    \end{pmatrix},$ $\boldsymbol{f}^x(\boldsymbol{q}) =  \begin{pmatrix}
         a_{*}p \\ 0 \\ a_{*}u    \end{pmatrix},$ $\boldsymbol{f}^y(\boldsymbol{q}) =  \begin{pmatrix}
         0 \\ a_{*}p \\ a_{*}v    \end{pmatrix}$ and ${\boldsymbol{s}(\boldsymbol{q})} =  \begin{pmatrix}
         cv \\ -cu \\ 0    \end{pmatrix}.$
    The system \eqref{eq:coriolis} can be equivalently written in the  form 
\begin{align}\label{eq:quasilinear}
    \frac{\partial}{\partial t} \boldsymbol{q} + {J}_x(\boldsymbol{q})  \frac{\partial}{\partial x} \boldsymbol{q} + {J}_y(\boldsymbol{q})  \frac{\partial}{\partial y} \boldsymbol{q} = S \boldsymbol{q},
\end{align} where  $S = \left(\begin{array}{ccc}
         0 & c & 0 \\ -c & 0 & 0 \\ 0 & 0 & 0
    \end{array}\right)$ and ${J}_x(\boldsymbol{q}) = J_x$ and ${J}_y(\boldsymbol{q}) = J_y$ are the globally constant (state-independent) Jacobian matrices for the linear system \eqref{eq:coriolis}, given by 
\begin{align}\label{eq:jacob_lin}
    J_x := \left(\begin{array}{ccc}
         0 & 0 & a_{*} \\ 0& 0& 0 \\ a_{*}& 0 & 0
    \end{array}\right), \quad \quad J_y := \left(\begin{array}{ccc}
         0 & 0 & 0 \\ 0& 0& a_{*} \\ 0& a_{*} & 0
    \end{array}\right).
\end{align} 

\subsection{Stationary solutions for the linear system}\label{subsec:stationary}

The existence of stationary solutions for a general two-dimensional linear hyperbolic system with source terms  \eqref{eq:quasilinear} can be understood using the Fourier expansion of the solution, see \cite{barsukow2019, Barsukow2025Fourier}. Solutions $\boldsymbol{q}:\mathbb{R}^{+} \times \mathbb{R}^2 \rightarrow \mathbb{R}^{m}$ to \eqref{eq:coriolis}, can be expressed as linear combinations of Fourier modes of the form 
\begin{align}\label{eq:fourier_mode}
\hat{\boldsymbol{q}}(t, \boldsymbol{k})\exp(\mathbb{i}\boldsymbol{k}\cdot \boldsymbol{x}), 
\end{align} for $\boldsymbol{k} \in \mathbb{R}^2,$ where $\boldsymbol{k}$ denotes the wave number vector which denotes the spatial frequency of the mode and $\hat{\boldsymbol{q}}(t, \boldsymbol{k})$  denotes the amplitude. To analyze the existence of stationary solutions, we substitute the mode \eqref{eq:fourier_mode} into \eqref{eq:quasilinear} which yields
\begin{align}\label{eq:fourier_analytic_sub}
    \frac{d}{dt} \hat{\boldsymbol{q}} + \mathbb{i} \boldsymbol{J} \cdot \hat{\boldsymbol{q}} = S \hat{\boldsymbol{q}},
 \end{align}
where $\boldsymbol{J} := (J_x, J_y).$ It follows from \eqref{eq:fourier_analytic_sub} that a Fourier mode \eqref{eq:fourier_mode} is stationary if $\hat{\boldsymbol{q}} \in \textrm{ker}(\mathbb{i}\boldsymbol{J}  \cdot \boldsymbol{k}  - S).$ Note that $\boldsymbol{k}$ determines the spatial behaviour of the mode. Stationary states of \eqref{eq:quasilinear} that require particular values of $\boldsymbol{k}$ are called \emph{trivial}: for them, no component of $\boldsymbol{q}$ can be chosen freely as a function of the independent spatial variable $\boldsymbol{x}$, (see \cite{barsukow2019}) for more details. For instance, for the system \eqref{eq:coriolis}, upon choosing $\boldsymbol{k} := (0,0)$ and $\hat{\boldsymbol{q}} = (0, \dots, 0) \in \mathbb{R}^m,$ one obtains \emph{a trivial stationary state}, which is globally constant. If $\mathrm{det}(\mathbb{i}\boldsymbol{J}  \cdot \boldsymbol{k}  - S) = 0$ for all $\boldsymbol{k} \in \mathbb{R}^2,$ then there exist non-trivial stationary states for the linear system \eqref{eq:coriolis}. This is indeed the case for the system \eqref{eq:coriolis}, since for any $\boldsymbol{k} \in \mathbb{R}^2,$ \begin{align}
    \det(\mathbb{i}\boldsymbol{J}\cdot\boldsymbol{k} - S) &= \begin{vmatrix} 0 & -c & \mathbb{i}k_x a_* \\ c & 0 & \mathbb{i}k_y a_* \\ \mathbb{i}k_x a_* & \mathbb{i}k_y a_* & 0 \end{vmatrix} \label{eq:analy_matr} =  c k_x k_y a_*^2 - c k_x k_y a_*^2 = 0.
\end{align}
The kernel of the matrix \eqref{eq:analy_matr} is spanned by  $\displaystyle(-\mathbb{i}k_y a_*,  \mathbb{i}k_x a_*, c)^{T}.$ For \eqref{eq:fourier_mode}  to be a nontrivial stationary solution, the vector $\hat{\boldsymbol{q}} = (\hat{u}, \hat{v}, \hat{p})^{T}$ has to be in the kernel, which is equivalent to $\displaystyle \hat{p}a_{*}\mathbb{i}k_y = -c \hat{u}, \, \hat{p}a_{*}\mathbb{i}k_x = c \hat{v}, \, k_x\hat{u} + k_y \hat{v} = 0,$ easily recognizable as the Fourier transforms of the relations
\begin{align}
        & a_{*}\nabla p = -c\boldsymbol{u}^{\perp}, \label{eq:geo_balance}\\
        & \nabla \cdot \boldsymbol{u}= 0. \label{eq:divfree}
    \end{align}
The  nontrivial  stationary state  described by the equations \eqref{eq:geo_balance}, \eqref{eq:divfree} for the linear system \eqref{eq:coriolis} is called a \emph{geostrophic equilibrium}, by analogy with the shallow water case \eqref{eq:geo_eq}-\eqref{geo_eq_divfree}. Observe that for the linear system \eqref{eq:coriolis}, it \emph{is} a stationary state. As before, the condition \eqref{eq:divfree} is a consequence of \eqref{eq:geo_balance}. \textcolor{black}{In the rest of the article, for the sake of simplicity, we consider the system \eqref{eq:coriolis} with $a_* = 1.$}

\section{The semi-discrete Active Flux method}\label{sec:afscheme}

In this section, we propose an Active Flux numerical method  to approximate the solutions to \eqref{eq:coriolis}. The Active Flux method evolves point values at the cell interfaces, in addition to the cell-averages. 
\subsection{Discretization and degrees of freedom}
The two-dimensional spatial domain $\mathbb{R}^2$ is discretized into a Cartesian grid with uniform mesh-sizes $\Delta x$ and $\Delta y,$ in the $x$ and $y$ directions, respectively: 
\begin{align*}
\displaystyle \mathbb{R}^2 &:= \bigcup_{i,j \in \mathbb{Z}^2} C_{ij}, & C_{ij} = \left[\left(i - \half\right)\Delta x, \left(i + \half\right)\Delta x\right] \times \left[\left(j - \half\right)\Delta y, \left(j + \half\right)\Delta y\right]. 
\end{align*}
At any given time $t \geq 0,$ we  consider nine degrees of freedom (DoF) for the cell $C_{ij}, \, i,j \in \mathbb{Z}$: (i) the cell-average $\displaystyle \bar{\boldsymbol{q}}_{i,j}(t),$ and (ii) the point value evaluations at the four nodes  $\boldsymbol{q}_{i \pm \half, j \pm \half}(t)$, and the four edge centers $\displaystyle \boldsymbol{q}_{ i \pm \half, j}(t), \boldsymbol{q}_{i, j \pm \half}(t).$ 
\par 

Since each nodal point value ($\boldsymbol{q}_{\iph, \jph}(t)$) is shared by four cells and each edge center point value ($\boldsymbol{q}_{\iph, j}(t)$ or $\boldsymbol{q}_{i, \jph}(t)$) is shared by two cells, the effective degrees of freedom corresponding to a cell is $ \left(4 \times \frac{1}{4}\right) + \left(4 \times \frac{1}{2}\right) + 1 = 4.$ For the sake of consistency, in the rest of the article, we set the four degrees of freedom belonging to the cell $C_{ij}$ to be the following:
\begin{align*} 
   \boldsymbol{q}^{N}_{ij}(t) &= \boldsymbol{q}_{\iph, \jph}(t), & \boldsymbol{q}^{E_{H}}_{ij}(t) &= \boldsymbol{q}_{i, \jph}(t), \\ \boldsymbol{q}^{E_{V}}_{ij}(t) &= \boldsymbol{q}_{\iph, j}(t), &\boldsymbol{q}^{A}_{ij}(t) &= \bar{\boldsymbol{q}}_{i, j}(t).
\end{align*} A visualization of a typical cell $C_{i,j},$ with the corresponding nine degrees of freedom (DoF) is given in Figure \ref{eq:af_dof}.\\
\textbf{Active Flux reconstruction.} 
 In each cell $C_{ij}, \, i,j \in \mathbb{Z},$  we consider the unique biparabolic reconstruction $\tilde{\boldsymbol{{q}}}_{ij}(\cdot, \cdot, t): C_{ij} \rightarrow \mathbb{R}^m,$  determined by the following nine relations: \begin{align*}
\frac{1}{\Delta x \Delta y} \int \int \tilde{\boldsymbol{q}}_{ij}(t) \dif x \dif y &= \bar{\boldsymbol{q}}_{i,j}(t), & \tilde{\boldsymbol{q}}_{ij}(x_{i \pm \half}, y_j, t)=\boldsymbol{\tilde{q}}_{i \pm \half, j, t},  \\  \tilde{\boldsymbol{q}}_{ij}(x_{i}, y_{j \pm \half}) &=  \boldsymbol{{q}}_{i, j \pm \half}(t), & \tilde{\boldsymbol{q}}_{i, j }(x_{i \pm \half}, y_{j \pm \half}, t)&=  \boldsymbol{{q}}_{i \pm \half, j \pm \half}(t).  \nonumber
\end{align*} Since the point values $\boldsymbol{{q}}_{\iph, j}(t), \boldsymbol{{q}}_{i, \jph}(t)$ and $ \boldsymbol{{q}}_{\iph, \jph}(t), \,\, i, j \in \mathbb{Z}$ are shared by adjacent cells, this leads to a globally continuous reconstruction $ \tilde{\boldsymbol{q}}(\cdot, \cdot, t) $ defined by
\begin{align}\label{eq:global_recon}
    \tilde{\boldsymbol{q}} (x, y, t) :=    \tilde{\boldsymbol{q}}_{ij} (x, y, t)  \quad \mbox{for} \, (x,y) \in C_{ij}.
\end{align} The reconstruction  \eqref{eq:global_recon} need not be differentiable at the Cartesian cell interfaces, i.e., the left and right traces of the spatial partial derivatives of $\boldsymbol{\tilde{\rho}}(\cdot, \cdot, t)$ may not coincide. 

\begin{figure}[htbp]
\centering
\begin{tikzpicture}[
    scale=0.8,
    >={Straight Barb[length=2mm]},
    open tri/.style={regular polygon, regular polygon sides=3, draw=myyellow, fill=white, thick, minimum size=8pt, inner sep=0pt},
    open red diam/.style={diamond, draw=myred, fill=white, thick, minimum size=8pt, inner sep=0pt},
    open cyan diam/.style={diamond, draw=mycyan, fill=white, thick, minimum size=8pt, inner sep=0pt},
    solid tri/.style={regular polygon, regular polygon sides=3, fill=myyellow, minimum size=8pt, inner sep=0pt},
    solid red diam/.style={diamond, fill=myred, minimum size=8pt, inner sep=0pt},
    solid cyan diam/.style={diamond, fill=mycyan, minimum size=8pt, inner sep=0pt},
    solid purp sq/.style={rectangle, fill=mypurple, minimum size=7pt, inner sep=0pt}
]
\centering

\draw[thick, myblue] (-2,-2) rectangle (2,2);

\draw[->, thick] (-1.2,-1.2) -- (-0.2,-1.2) node[above, font=\large] {$x$};
\draw[->, thick] (-1.2,-1.2) -- (-1.2,-0.2) node[right, font=\large] {$y$};

\node[solid purp sq] (n0) at (0,0) {};
\node[above, text=mypurple, yshift=3pt] at (n0) {$q^A$};
\node[below right, font=\small, xshift=1pt, yshift=-1pt] at (n0) {0};

\node[open tri] (n1) at (-2,-2) {};
\node[above right, font=\small, xshift=1pt, yshift=1pt] at (n1) {1};

\node[open red diam] (n2) at (0,-2) {};
\node[above, font=\small, yshift=3pt] at (n2) {2};

\node[open tri] (n3) at (2,-2) {};
\node[above left, font=\small, xshift=-1pt, yshift=1pt] at (n3) {3};

\node[solid cyan diam] (n4) at (2,0) {};
\node[left, font=\small, xshift=-3pt] at (n4) {4};
\node[right, text=mycyan, xshift=3pt] at (n4) {$q^{E_V}$};

\node[solid tri] (n5) at (2,2) {};
\node[below left, font=\small, xshift=-1pt, yshift=-1pt] at (n5) {5};
\node[above left, text=myyellow, xshift=-1pt, yshift=2pt] at (n5) {$q^N$};

\node[solid red diam] (n6) at (0,2) {};
\node[below, font=\small, yshift=-3pt] at (n6) {6};
\node[above, text=myred, yshift=3pt] at (n6) {$q^{E_H}$};

\node[open tri] (n7) at (-2,2) {};
\node[below right, font=\small, xshift=1pt, yshift=-1pt] at (n7) {7};

\node[open cyan diam] (n8) at (-2,0) {};
\node[right, font=\small, xshift=3pt] at (n8) {8};

\end{tikzpicture}
\caption{The four degrees of freedom $ \boldsymbol{q}^{N}, \, \boldsymbol{q}^{E_{H}}, \, \boldsymbol{q}^{E_{V}}, \, \mbox{and} \, \boldsymbol{q}^{A}$ corresponding to a cell $C_{ij}$ for the Active Flux method in a 2-D Cartesian cell.}
\label{eq:af_dof}
\end{figure}
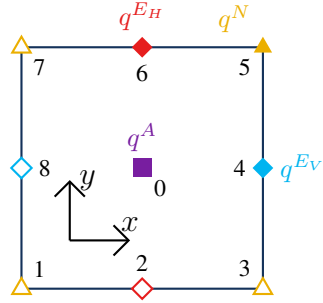
\subsection{Semi-discrete numerical method}
We are interested in a semi-discrete formulation of Active Flux, originally proposed in \cite{abgrall2023b, abgrall_barsukow2023}. The main idea is to derive semi-discrete evolution formulas for the cell-average using the conservative formulation \eqref{eq:coriolis} and the point values using the quasilinear formulation \eqref{eq:quasilinear}, resulting in a system of ordinary differential equations for each of the degrees of freedom. These are then evolved in time using a suitable Strong Stability Preserving (SSP) Runge Kutta time integrator \cite{shu1988, gottlieb1998otalVD}. 
\subsubsection{Evolution of cell-averages}
Integrating the system \eqref{eq:coriolis} in the cell $C_{ij}$ yields the semi-discrete formula 
\begin{align}\label{eq:af_cellav_evo}
    \frac{d}{dt} \bar{\boldsymbol{q}}_{i,j}(t) &= {-\frac{1}{\Delta x} \left( \boldsymbol{f}^x_{i+\frac{1}{2},j} - \boldsymbol{f}^x_{i-\frac{1}{2},j} \right) -\frac{1}{\Delta y} \left( \boldsymbol{f}^y_{i,j+\frac{1}{2}} - \boldsymbol{f}^y_{i,j+\frac{1}{2}} \right)} + {\bar{\boldsymbol{s}}_{ij}}, 
    \end{align}
for the evolution of the cell-average values $\displaystyle \bar{\boldsymbol{q}}_{ij} :=\frac{1}{\Delta x \Delta y} \int_{y_{\jmh}}^{y_{\jph}} \int_{x_{\imh}}^{x_{\iph}} \boldsymbol{q}(x,y,t) \dif x \dif y,$ where $\displaystyle \bar{\boldsymbol{s}}_{ij} := \frac{1}{\Delta x \Delta y} \int_{y_{\jmh}}^{y_{\jph}} \int_{x_{\imh}}^{x_{\iph}} \boldsymbol{s}(\boldsymbol{q}(x,y,t)) \dif x \dif y$ and \begin{align*}
    \displaystyle \boldsymbol{f}^x_{i+\frac{1}{2},j}:= \frac{1}{6}\boldsymbol{f}^x(q_{ij}^N) + \frac{4}{6}\boldsymbol{f}^x(q_{ij}^{E_V}) +  \frac{1}{6}\boldsymbol{f}^x(q_{i,j-1}^N) \approx \int_{y_{\jmh}}^{y_{\jph}}\boldsymbol{f}(q(x_{\iph}, y)) \dif y,\\
    \displaystyle \boldsymbol{f}^y_{i,j+\frac{1}{2}}:= \frac{1}{6}\boldsymbol{f}^y(q_{i-1,j}^N) + \frac{4}{6}\boldsymbol{f}^y(q_{ij}^{E_H}) +  \frac{1}{6}\boldsymbol{f}^y(q_{i,j}^N) \approx \int_{x_{\imh}}^{x_{\iph}}\boldsymbol{f}(q(x, y_{\jph})) \dif x, \end{align*} are the approximations to the flux integrals on the interfaces computed using the Simpson's quadrature rule.  \textcolor{black}{We note that for source terms $s(\boldsymbol{q})$ considered in this article, which are linear functions of $\boldsymbol{q}$ ($s((u, v, p)^{T}) = (cv, -cu, 0)^{T}$ for \eqref{eq:coriolis} and $s((h, hu, hv)^{T}) = (0, \frac{1}{R_0}hv, -\frac{1}{R_0}hu)^{T}$ in \eqref{eq:swe_noncons}) the source term cell-average  $\displaystyle \bar{s}_{ij}$ can be computed easily as linear combinations of $\bar{\boldsymbol{q}}_{ij}.$} Unlike usual finite volume methods, owing to the continuity of the solution reconstruction \eqref{eq:global_recon} across the cell interfaces, no Riemann solver is required in the cell-average evolution formula \eqref{eq:af_cellav_evo}.
    
\subsection{Evolution of point values}
Next, for the update of the point values $\boldsymbol{q}_{i,j}^P$ where $P \in \{E_H, N, E_V\},$ fixing $(x,y) = P,$ from the quasilinear formulation \eqref{eq:quasilinear}, we can write 
\begin{align}\label{eq:quas_pt}
    \frac{d}{dt} \boldsymbol{q}^{P} = -{J}_x(\boldsymbol{q}^{P})  \frac{\partial}{\partial x} \boldsymbol{q}^{P} - {J}_y(\boldsymbol{q}^{P})  \frac{\partial}{\partial y} \boldsymbol{q}^{P} + S(\boldsymbol{q}^{P}).
\end{align} 
{To approximate the partial derivatives $\frac{\partial}{\partial x} \boldsymbol{q}^{P},\frac{\partial}{\partial y} \boldsymbol{q}^{P}$ appearing in \eqref{eq:quas_pt}, we consider the reconstruction $\tilde{\boldsymbol{q}} $ defined in \eqref{eq:global_recon}.}
The property that the left and right traces of the derivatives of  $\tilde{\boldsymbol{q}} $ need not be equal, can  be used to induce upwinding in the formula \eqref{eq:quas_pt}. This can be achieved (also see \cite{abgrall_barsukow2023, Barsukow2025Fourier}) by splitting the positive and negative components of the Jacobian matrices  $J_x = \partial \boldsymbol{f}^x/\partial \boldsymbol{q}$ and $J_y= \partial \boldsymbol{f}^y/\partial \boldsymbol{q},$ to act on the left and right traces of the reconstruction \eqref{eq:global_recon},  yielding a semi-discrete point value evolution formula: 
\begin{align}\label{eq:af_ptvalue}
    \frac{d}{dt} \boldsymbol{q}^P &= - J_x^{+}(\boldsymbol{q}^P)      D_x^{+} 
          (\boldsymbol{q}^P)
     -  J_x^{-}(\boldsymbol{q}^P)     D_x^-(\boldsymbol{q}^P) \\& \spc - J_y^{+} (\boldsymbol{q}^P)D_y^+(\boldsymbol{q}^P) -  J_y^{-}(\boldsymbol{q}^P)D_y^-(\boldsymbol{q}^P)  + \boldsymbol{s}(\boldsymbol{q}^{P}), \nonumber
\end{align} with the splitting $J_x  = J_x^{+} + J_x^{-}, J_y  = J_y^{+} + J_y^{-}$ where $J_x^{\pm}$ and $J_y^{\pm}$ are the positive and negative components $J_x$ and $J_y$ respectively, and $D_x^{\pm}\boldsymbol{q}^P$     and $D_y^{\pm}\boldsymbol{q}^P$ are the traces of the  $x$ and $y$ derivatives of the reconstruction \eqref{eq:global_recon} respectively, at the point $P,$ see \cite{Barsukow2025Fourier,
barsukow_lisa_2026} for details.
\par
Using the diagonalization of the Jacobian matrices $$J_x = R_{x} \mathrm{diag}(\lambda_1^x, \dots, \lambda_m^x) R^{-1}_{x}, \quad J_y = R_{y} \mathrm{diag}(\lambda_1^y, \dots, \lambda_m^y) R^{-1}_{y},$$ we compute the positive and negative matrices $J_x^{\pm}, J_y^{\pm}$ in \eqref{eq:af_ptvalue} as in \cite{Barsukow2025Fourier} as follows: 
    \begin{align}\label{eq:upwind_js}
        J_x^{\pm} &=  R_{x} \mathrm{diag}\left((\lambda_1^{x})^{\pm}, \dots, (\lambda_m^x)^{\pm}\right) R^{-1}_{x}, \,\,  J_y^{\pm} =  R_{y} \mathrm{diag}\left((\lambda_1^{y})^{\pm}, \dots, (\lambda_m^y)^{\pm}\right) R^{-1}_{y},
    \end{align} where $a^+ := \max\{0, a\}$ and $a^- = \min\{0, a\}$ for $a \in \mathbb{R}.$ 
    \par 
Now, substituting the Jacobian matrices \eqref{eq:jacob_lin} corresponding to the system \eqref{eq:coriolis} in \eqref{eq:af_ptvalue}, we obtain the following semi-discrete evolution formula for the point values:
\begin{align}\label{eq:ptupdate}
    \frac{d}{dt} \begin{pmatrix}
         u \\ v \\ p
    \end{pmatrix}_{(x,y) = P} &  = {\frac{1}{2}\begin{pmatrix}
        -(D_x^+ + D_{x}^{-})p \\
        -(D_y^+ + D_{y}^{-})p \\
        -(D_x^{+}+ D_x^{-} )u - (D_{y}^{+}+ D_{y}^{-})v 
    \end{pmatrix}_{(x,y) = P}} + {\begin{pmatrix}
    cv \\ -cu \\0 
    \end{pmatrix}_{(x,y) = P}}  \\ & \spc + \underbrace{{\frac{1}{2}\begin{pmatrix}
        -(D_x^+ - D_{x}^{-})u \\
        -(D_y^+ - D_{y}^{-})v \\
        -(D_x^{+}- D_x^{-} )p - (D_{y}^{+}- D_{y}^{-})p 
    \end{pmatrix}_{(x,y) = P}}}_{\mathscr D_P}. \nonumber
\end{align}  
\begin{remark}\label{remark:central_af}(Central Active Flux method)
    In the semi-discrete point value evolution formula \eqref{eq:ptupdate}, the term 
    $\mathscr D_P$ acts as numerical diffusion. Without $\mathscr D_P$, the semi-discrete formula \eqref{eq:ptupdate}  together with \eqref{eq:af_cellav_evo} constitutes the \emph{central Active Flux} method (see \cite{Barsukow2025Fourier, BarsukowBIT2026}). 
\end{remark}

\begin{remark}
    An alternative to the Jacobian splitting given in \eqref{eq:upwind_js} is to consider $\displaystyle a_x := \max_{i \in \{1,2, \dots, m\}}\{\abs{\lambda(J_x)}\}, \, a_y :=  \max_{i \in \{1,2, \dots, m\}}\{\abs{\lambda(J_y)}\},$ using which we can define the \emph{Rusanov splitting} (see \cite{Barsukow2025Fourier}): 
    \begin{align}\label{eq:rusanov}
        J_x^{\pm} = \frac{1}{2} \left( J_x \pm a_x \mathbb{1}_m \right), \quad J_y^{\pm} = \frac{1}{2} \left(J_y \pm a_y \mathbb{1}_m\right),
    \end{align} where $\mathbb{1}_m$ is the $m \times m$ identity matrix. However, as was in the case for homogenous linear acoustics system in \cite{barsukow2025b}, this splitting yields an Active Flux method which, as we will observe in Section \ref{sec:numericalexp}, is not stationarity preserving.
\end{remark}


\subsection{Discretization in time} The system \eqref{eq:af_cellav_evo}, \eqref{eq:af_ptvalue} of ordinary differential equations can be evolved in time using the strong stability preserving Runge Kutta (SSP-RK3) third-order time integration \cite{shu1988, gottlieb1998otalVD, gottliebsigalshu2001}. The time-step $\Delta t$ needs to satisfy a Courant-Friedrichs-Lewy (CFL) condition\begin{equation}\label{eq:cfl}
    \Delta t \leq C_\text{CFL} \, \frac{\min\{\Delta x, \Delta y\}}{\lambda_{\max}},
\end{equation} where $\lambda_{\max}$ represents the maximum spectral radius of the flux Jacobians  $ \partial \boldsymbol{f}^x/\partial \boldsymbol{q}$ and $\partial \boldsymbol{f}^y/\partial \boldsymbol{q}.$ The analysis in Section \ref{sec:fourier} reveals that $C_\text{CFL} \leq 0.27$ for the fully-discrete method to remain stable in two-spatial dimensions. 
\section{Stationarity preservation} \label{sec:fourier}

Fourier analysis of numerical methods for linear PDE  systems \textcolor{black}{on Cartesian grids} is used to ensure that under a suitable time-step restriction, the modes of the numerical solution are non-increasing and remain stable. Further, Fourier modes that remain stationary under the application of the numerical method can be used to characterize its discrete stationary states. \textcolor{black}{The stationarity preservation of linear numerical methods on Cartesian grids have been analyzed using the discrete Fourier transform in  \cite{barsukow2019, barsukow2023, Barsukow2025Fourier, Barsukow2025Stationarity}.} Often, most numerical methods are only capable of preserving  trivial stationary states (refer to Section \ref{subsec:stationary} for the definition). A numerical method is said to be \emph{stationarity preserving} for a system of PDEs, if it discretizes all the analytical stationary states of the PDE (\cite{barsukow2019}). From the discussion in Section \ref{subsec:stationary}, for the system \eqref{eq:coriolis} this means that a stationarity preserving numerical method should admit some discretization of the geostrophic equilibrium \eqref{eq:geo_balance} which is kept stationary by the method.
\subsection{Stationarity preservation of semi-discrete Active Flux using the discrete Fourier transform}
  We apply a discrete spatial Fourier transform to the semi-discrete Active Flux formulas \eqref{eq:af_cellav_evo}, \eqref{eq:af_ptvalue}. Consider a discrete grid function $q_{i,j}^X(t)$ representing a degree of freedom of type $X \in \{A, E_H, E_V, N\}$ in cell $(i,j)$ with grid spacings $\Delta x$ and $\Delta y$. 
 This spatial data can be expressed as a superposition of discrete Fourier modes, as follows
\begin{equation}
    q_{i,j}^X(t) = \sum_{\boldsymbol{k}} \hat{q}^X(t, \boldsymbol{k}) \exp(\mathbb{i} (k_x i\Delta x + k_y j\Delta y)), \quad X \in \{A, E_H, E_V, N\},
\end{equation}
where $\boldsymbol{k} = (k_x, k_y)$ is the spatial frequency (wave number vector) and the coefficient $\hat{q}^X(t,\boldsymbol{k}) \in \mathbb{R}^m$ denotes the \emph{discrete Fourier transform} of the  function $q^X$. To simplify the notation, we define the {translation factors}: $
    t_x = \exp(\mathbb{i} k_x \Delta x),   t_y = \exp(\mathbb{i} k_y \Delta y).$ Now, any spatial shift in the grid stencil corresponds simply to multiplying the Fourier coefficient $\hat{q}^X$ by the appropriate powers of $t_x$ and $t_y,$ i.e., 
\begin{equation}
    q_{i+k,j+l}^X(t) = \sum_{\boldsymbol{k}} \left(t_x^k t_y^l\hat{q}^X(t, \boldsymbol{k}) \right) \exp(\mathbb{i} (k_x i\Delta x + k_y j\Delta y)),
\end{equation}
see \cite{Barsukow2025Fourier}, Section 4 for more details. Since \eqref{eq:coriolis} consists of three variables $(u, v, p)$, and in the Active Flux method \eqref{eq:af_cellav_evo}, \eqref{eq:af_ptvalue} we have four types of spatial degrees of freedom corresponding to each cell (one cell average, one horizontal edge, one vertical edge, and one node), the full discrete state of a cell is described by $12$ degrees of freedom.  When we substitute this Fourier ansatz into the semi-discrete evolution formulas  \eqref{eq:af_cellav_evo}, \eqref{eq:af_ptvalue}, the global exponential factors drop out of the equations, yielding a linear system of ordinary differential equations
\begin{equation}\label{eq:evolutionmatrix}
    \frac{d}{dt} \boldsymbol{\hat{q}} + \mathcal{E}(t_x, t_y)\boldsymbol{\hat{q}} = 0,
\end{equation} where $\displaystyle  \boldsymbol{\hat{q}} = (\hat{u}^{A}, \hat{v}^{A}, \hat{p}^{A}, \hat{u}^{E_H}, \hat{v}^{E_H}, \hat{p}^{E_H}, \hat{u}^{E_V}, \hat{v}^{E_V}, \hat{p}^{E_V}, \hat{u}^{N}, \hat{v}^{N}, \hat{p}^{N})^T,$ with a $12 \times 12$ complex-valued dense \emph{evolution matrix} $\mathcal{E}(t_x,t_y)$. For simplicity, we denote the matrix by $\mathcal{E}$ and represent it in the form
\begin{equation}\label{eq:evolutionmatrix_appx}
\mathcal{E} = 
\begin{bmatrix}
\mathcal{E}_{AA} & \mathcal{E}_{A E_H} & \mathcal{E}_{A E_V} & \mathcal{E}_{AN} \\
\mathcal{E}_{E_H A} & \mathcal{E}_{E_H E_H} & \mathcal{E}_{E_H E_V} & \mathcal{E}_{E_H N} \\
\mathcal{E}_{E_V A} & \mathcal{E}_{E_V E_H} & \mathcal{E}_{E_V E_V} & \mathcal{E}_{E_V N} \\
\mathcal{E}_{NA} & \mathcal{E}_{N E_H} & \mathcal{E}_{N E_V} & \mathcal{E}_{NN},
\end{bmatrix},
\end{equation}
where each submatrix $\mathcal{E}_{X_1 X_2}$ denotes the effect of $\hat{q}^{X_2}$ on the evolution of $\hat{q}^{X_1}$ for $X_{1}, X_2 \in \{A, E_H, E_V, N\}.$ The submatrices can be obtained trivially from \eqref{eq:af_cellav_evo}-\eqref{eq:af_ptvalue} (see for eg., Appendix C in  \cite{Barsukow2025Fourier}).  Note that \eqref{eq:evolutionmatrix} can be seen as a discrete analogue of the expression \eqref{eq:fourier_analytic_sub}, with matrix $\mathcal{E}$ being the discrete counterpart of $\mathbb{i}\boldsymbol{J}  \cdot \boldsymbol{k}  - S.$ 
From \eqref{eq:evolutionmatrix}, it is immediate that the kernel of the matrix $\mathcal{E}$ corresponds to the Fourier transform of the discrete datum that is kept stationary by the numerical method. For the method to be stationarity preserving, the kernel of \eqref{eq:evolutionmatrix_appx} has to be non-trivial for any choice of $t_x, t_y$ (i.e., of $\boldsymbol{k}$) \cite{barsukow2019, Barsukow2025Fourier}. Due to the excessive length of computations, the kernel of the matrix $\mathcal{E}$ is evaluated using MATHEMATICA and is found to be non-trivial, as described in the following theorem.

\begin{theorem}\label{thm:stationarity_af}
    The Active Flux evolution matrix $\mathcal{E}$ in \eqref{eq:evolutionmatrix} satisfies $\mathrm{det}(\mathcal{E}) = 0$ for all wave numbers $\boldsymbol{k} \in \mathbb{R}^2.$ The kernel of $\mathcal{E}$ is the one-dimensional space  spanned by the vector 
\begin{equation}\label{eq:af_kernel}
\resizebox{\linewidth}{!}{%
$\begin{aligned}
 \Bigg(
 &\frac{(1+4t_x+t_x^2)(1-t_y)}{3 c \, \Delta y \, t_x(1+t_x)t_y}, && \frac{(t_x-1)(1+4t_y+t_y^2)}{3 c \, \Delta x \, t_x \, t_y(1+t_y)}, && \frac{(1+4t_x+t_x^2)(1+4t_y+t_y^2)}{9t_x(1+t_x)t_y(1+t_y)}, \\[2ex]
 &-\frac{(1+6t_x+t_x^2)(t_y-1)}{2 c \, \Delta y \, t_x(1+t_x)(1+t_y)}, && \frac{t_x-1}{c \, \Delta x \, t_x}, && \frac{1+6t_x+t_x^2}{4t_x+4t_x^2}, \\[2ex]
 & \frac{1-t_y}{c \, \Delta y \, t_y}, && \frac{(t_x-1)(1+6t_y+t_y^2)}{2 c \, \Delta x(1+t_x)t_y(1+t_y)}, && \frac{1+6t_y+t_y^2}{4t_y+4t_y^2}, \\[2ex]
 & \frac{2(1-t_y)}{c \, \Delta y + c \, \Delta y \, t_y}, && -\frac{2(1-t_x)}{c \, \Delta x + c \, \Delta x \, t_x}, && \quad \quad 1 
 \Bigg).
\end{aligned}$%
}
\end{equation}
\end{theorem}

\begin{corollary}
   The semi-discrete Active Flux method \eqref{eq:af_cellav_evo}, \eqref{eq:af_ptvalue} is stationarity preserving for the two-dimensional linear acoustics system with Coriolis source terms \eqref{eq:coriolis}. The discrete stationary state of the method is characterized by the following relations
\begin{align}
    &c \frac{u^A_{i, j-1} + 4u^A_{i, j} + u^A_{i, j+1}}{6} + \frac{p^A_{i, j+1} - p^A_{i, j-1}}{2\Delta y} = 0, \label{eq:res1}\\[6pt]
    &c \frac{v^A_{i-1, j} + 4v^A_{i, j} + v^A_{i+1, j}}{6} - \frac{p^A_{i+1, j} - p^A_{i-1, j}}{2\Delta x} = 0, \label{eq:res2}\\[6pt]
    &c v^{E_H}_{i, j} - \frac{p^N_{i, j} - p^N_{i-1, j}}{\Delta x} = 0, \quad
    c u^{E_V}_{i, j} + \frac{p^N_{i, j} - p^N_{i, j-1}}{\Delta y} = 0, \label{eq:res4}\\[6pt]
    &c \frac{u^N_{i, j+1} + u^N_{i, j}}{2} + \frac{p^N_{i, j+1} - p^N_{i, j}}{\Delta y} = 0, \quad
    c \frac{v^N_{i+1, j} + v^N_{i, j}}{2} - \frac{p^N_{i+1, j} - p^N_{i, j}}{\Delta x} = 0, \label{eq:res6}\\[6pt]
    &c \frac{u^{E_H}_{i, j+1} + u^{E_H}_{i, j}}{2} + \frac{p^{E_H}_{i, j+1} - p^{E_H}_{i, j}}{\Delta y} = 0, \quad c \frac{v^{E_V}_{i+1, j} + v^{E_V}_{i, j}}{2} - \frac{p^{E_V}_{i+1, j} - p^{E_V}_{i, j}}{\Delta x} = 0,\label{eq:res8}\\[6pt]
    & 4 (p^{E_V}_{i,j} + p^{E_V}_{i,j+1}) - p^{N}_{i, j-1}- 6 p^{N}_{i, j} - p^{N}_{i, j+1} = 0, \label{eq:res9}\\[6pt]
    & 4 (p^{E_H}_{i,j} + p^{E_H}_{i+1,j}) - p^{N}_{i-1, j}- 6 p^{N}_{i, j} - p^{N}_{i+1, j} = 0, \label{eq:res10}\\[6pt]
    & 9(p^{A}_{i,j} + p^{A}_{i+1,j} + p^{A}_{i,j+1}+ p^{A}_{i+1, j+1}) - p^{N}_{i-1,j-1} - 4p^{N}_{i-1,j} - p^{N}_{i-1,j+1}  \label{eq:res11} \\  & \spc -4p^{N}_{i,j-1}  - 16 p^{N}_{i,j}  - 4 p^{N}_{i,j+1} - p^{N}_{i+1,j-1} -4p^{N}_{i+1,j} - p^{N}_{i+1,j+1} = 0.  \no
\end{align}  
\end{corollary}

\begin{proof} 
Let $\hat{\boldsymbol{q}}$ be the discrete Fourier transform of a discrete stationary solution $\{q^{X}_{ij}\}_{i,j \in \mathbb{Z}}$. Its Fourier transform $\hat{\boldsymbol{q}}$ is proportional to the basis vector given in \eqref{eq:af_kernel}, by Theorem \ref{thm:stationarity_af}. 
The ratio of the first and third components of the vector in \eqref{eq:af_kernel} is
\begin{equation*}
    \frac{\hat{u}^A}{\hat{p}^A}  = \frac{-\frac{(1+4t_x+t_x^2)(t_y-1)}{3c\Delta y t_x (1+t_x)t_y}}{\frac{(1+4t_x+t_x^2)(1+4t_y+t_y^2)}{9 t_x(1+t_x)t_y(1+t_y)}}= -\frac{3(t_y^2-1)}{c\Delta y (1+4t_y+t_y^2)}.
\end{equation*}
Cross-multiplying and dividing by $6 t_y \Delta y$ yields
\begin{equation}
    c \hat{u}^A \left(\frac{t_y^{-1} + 4 + t_y}{6}\right) + \hat{p}^A \left(\frac{t_y - t_y^{-1}}{2\Delta y}\right) = 0.
\end{equation}
Now, applying the inverse discrete Fourier transform yields \eqref{eq:res1}, which is a discretization of $\displaystyle cu+\frac{\partial}{\partial y} p = 0,$ in \eqref{eq:geo_balance}. The remaining formulas can be derived similarly.

\end{proof}

\begin{remark}\label{remark:central_af_diff} For the central Active Flux method described in Remark \ref{remark:central_af}, a similar  Fourier analysis as above yields that the kernel of the  matrix is one-dimensional and is the same as that of the upwind Active Flux \eqref{eq:af_cellav_evo}, \eqref{eq:ptupdate}  as given by \eqref{eq:af_kernel}. This implies that the numerical diffusion in the point value evolution formula \eqref{eq:af_ptvalue} vanishes on the kernel \eqref{eq:af_kernel} of the upwind/central Active Flux matrix. 
    
\end{remark}

\subsection{Dispersion analysis}
The presence of spurious oscillatory modes has been studied for the system \eqref{eq:coriolis} in \cite{bernard2008, leroux2012}. In this section, we analyze the dispersion behaviour of the Active Flux method. Similar to the approach in \cite{bernard2008}, this can be done by analyzing the imaginary parts of the eigenvalues of the matrix $\mathcal{E}$ in the semi-discrete evolution formula \eqref{eq:evolutionmatrix}.

For a given wave number vector $\boldsymbol{k} = (k_x, k_y)$, consider the 12 complex eigenvalues $\lambda_j(\boldsymbol{k})$ of the matrix $\mathcal{E}$ given in  \eqref{eq:evolutionmatrix}:
\begin{equation}
    \lambda_j(\boldsymbol{k}) = \sigma_j(\boldsymbol{k}) + \mathbb{i}\omega_j(\boldsymbol{k}), \quad j \in \{1, 2, \dots, 12\}.
\end{equation}
Here, the imaginary part $\omega_j(\boldsymbol{k}) = \text{Im}(\lambda_j)$ represents the discrete wave frequencies supported by the method (numerical dispersion), and the real part $\sigma_j(\boldsymbol{k}) = \text{Re}(\lambda_j)$ represents the amplitude decay of the corresponding mode (numerical dissipation). These numerical modes are compared against the exact analytical frequencies (refer to the derivation in Section 2.3 of \cite{bernard2008}) given by $\omega_{\text{exact}} \in \{0, \pm\sqrt{k_x^2 + k_y^2 + c^2}\}$. We assume that $k_x = k_y = k$ and set $\Delta x = \Delta y = 1.$

\par

The imaginary parts of the 12 eigenvalues of the Active Flux evolution matrix $\mathcal{E}$ are displayed in Figure \ref{fig:von_neumann_analysis}(a). \textcolor{black}{Comparing with the 3 known analytical modes (also displayed in Figure \ref{fig:von_neumann_analysis}(a)), 3 out of the 12 modes can be identified as physical.  In Figure \ref{fig:von_neumann_analysis}(b), the absolute dispersion errors of the 3 physical Active Flux modes in comparison with the exact analytical modes are displayed.  Since the continuous physical equations are non-dissipative ($\sigma_{\text{exact}} = 0$), the absolute dissipation error is the magnitude of the real part of the eigenvalues, $|\sigma_j(k)|$. In Figure \ref{fig:von_neumann_analysis}(c) the absolute dissipation errors of the 12 Active Flux modes are shown. The physical modes show an error convergence rate of order $4$ for both the dispersion and dissipation errors  in Figure \ref{fig:von_neumann_analysis}(b) and  Figure \ref{fig:von_neumann_analysis}(c), respectively. Further, Figure \ref{fig:von_neumann_analysis}(a) indicates the presence of modes other than the physical ones, but the method assigns correspondingly high numerical diffusion to such modes as seen in Figure \ref{fig:von_neumann_analysis}(c), indicating that they will be damped.}

\begin{figure}[htbp]
    \centering

    \begin{subfigure}[b]{0.75\textwidth}
        \centering
        \includegraphics[width=\textwidth]{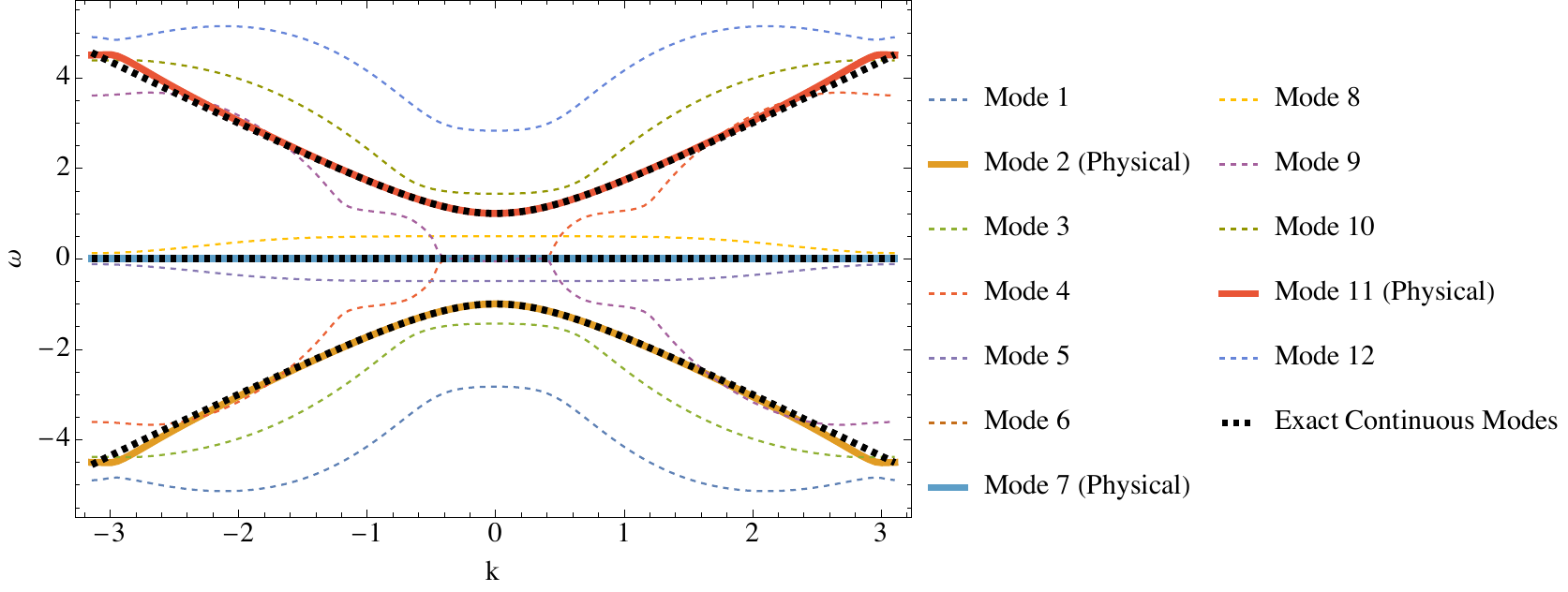}
        \caption{Wave Dispersion}
        \label{fig:vn_dispersion}
    \end{subfigure}

    \vspace{0.5em}

    \begin{subfigure}[b]{0.45\textwidth}
        \centering
        \includegraphics[width=\textwidth]{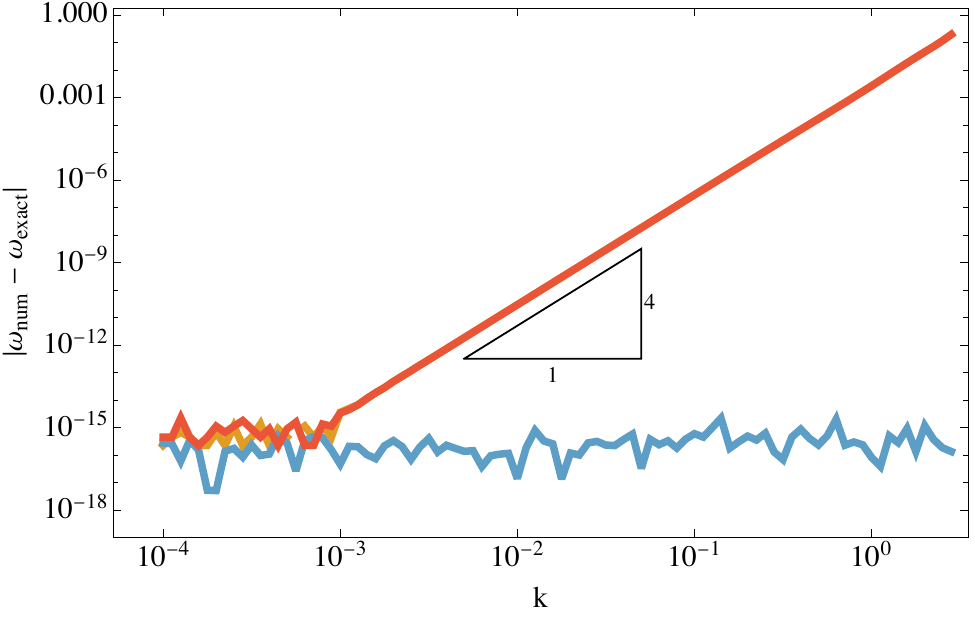}
        \caption{Dispersion Error (physical modes)}
        \label{fig:vn_dispersion_error}
    \end{subfigure}
    \hfill
    \begin{subfigure}[b]{0.45\textwidth}
        \centering
        \includegraphics[width=\textwidth]{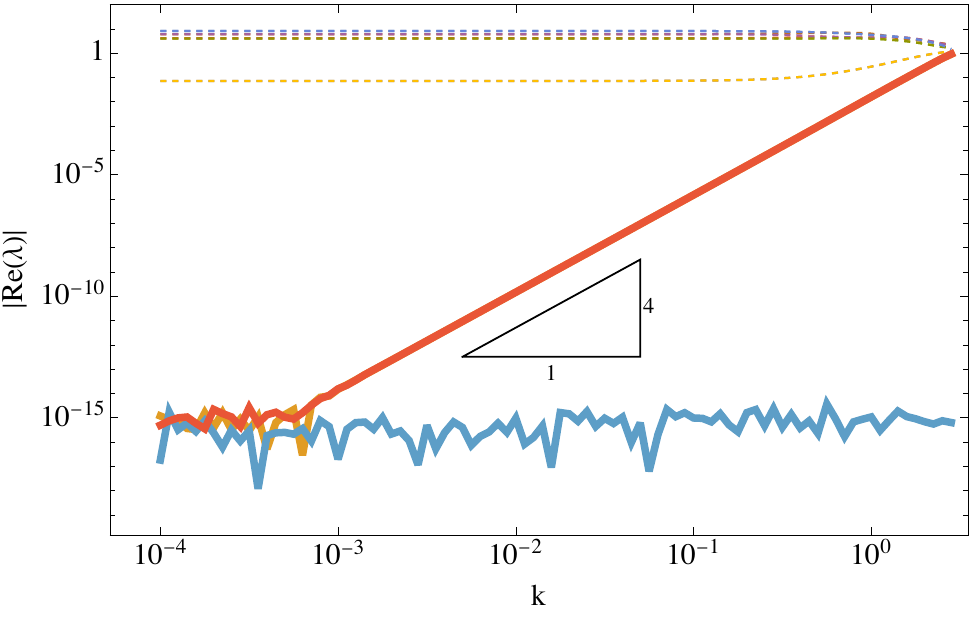}
        \caption{Dissipation Error}
        \label{fig:vn_dissipation}
    \end{subfigure}

    \caption{For each wave number $k = k_x = k_y$, let
    $\lambda_i(k) = \sigma_i(k) + \mathbb{i}\omega_i(k)$,
    $i \in \{1,2,\dots,12\}$, be the corresponding eigenvalues of
    the matrix $\mathcal{E}$ in \eqref{eq:evolutionmatrix} for
    $\Delta x = \Delta y = 1$ and Coriolis parameter $c=1$.
    (a) The numerical dispersion (the imaginary part
    $\omega_i(k)$) plotted against $k$. (b) The absolute dispersion
    error $|\omega_{\mathrm{numerical}}-\omega_{\mathrm{exact}}|$
    of the physical numerical modes with respect to the known exact
    analytical modes. (c) The numerical dissipation error
    $|\operatorname{Re}(\lambda(k))|=|\sigma(k)|$.}
    \label{fig:von_neumann_analysis}
\end{figure}
\subsection{Fully discrete stability analysis}
For an analysis of numerical diffusion, following Section 6 of \cite{Barsukow2025Fourier}, the Fourier transform of the fully discrete Active Flux method is given by
\begin{equation}
    \hat{q}^{n+1} = \mathcal{A}(s, \varphi)\hat{q}^n,
\end{equation}
with the amplification matrix of the third-order Runge-Kutta method given by
\begin{equation}\label{eq:amplific_mat}
    \mathcal{A} = \mathbf{I} - \Delta t \mathcal{E} + \frac{1}{2}\Delta t^2 \mathcal{E}^2 - \frac{1}{6}\Delta t^3 \mathcal{E}^3,
\end{equation} where $\mathcal{E}$ is the matrix in \eqref{eq:evolutionmatrix}.
The vector $\boldsymbol{k} = (k_x, k_y)^T$ is parametrized using $s \in [-\pi, \pi]$ and $\varphi$ as $\boldsymbol{k} = s (\cos \varphi, \sin \varphi)^{T}.$
Figure \ref{fig:fourier_analysis} shows the absolute values of the 12 eigenvalues of $\mathcal{A}$ as a function of the wave number $s$ for different time steps $\Delta t$ and $\varphi$. We set $\Delta x = \Delta y = 1$ and the Coriolis parameter $c = 1$. Von Neumann stability requires that the spectral radius satisfies $\rho(\mathcal{A}) \le 1$. As observed in Figure \ref{fig:fourier_analysis}, a time step up to $\Delta t = 0.275$ remains stable for all wave numbers. However, for time steps $\Delta t \geq 0.29$, the method violates the stability threshold. 

\begin{figure}[htbp]
     \centering
     \includegraphics[width=0.8\textwidth]{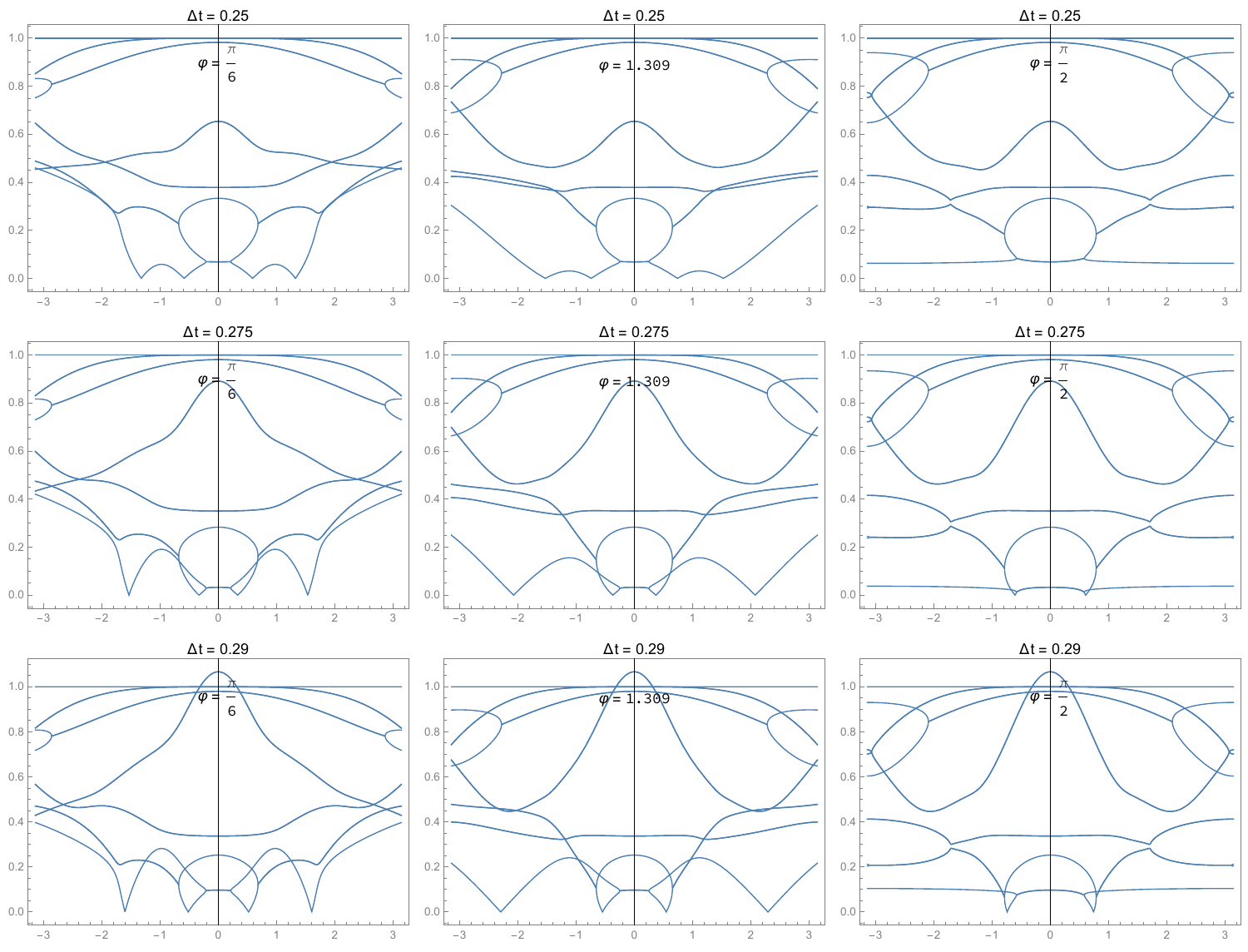}
     \caption{The absolute value of the eigenvalues of the amplification matrix \eqref{eq:amplific_mat} for $\Delta x = \Delta y =1, \, c=1,$  is plotted as a function of $s \in [-\pi,\pi]$ for different time steps $\Delta t$ (rows) and $\varphi$ (columns).}
     \label{fig:fourier_analysis}
\end{figure}

\section{Numerical experiments}\label{sec:numericalexp}
We present a series of numerical test cases to evaluate the performance of the semi-discrete Active Flux method \eqref{eq:af_cellav_evo}, \eqref{eq:af_ptvalue}. Throughout all experiments, $C_{\text{CFL}}$ in \eqref{eq:cfl} is set to $0.27$, in accordance with the Fourier stability analysis of Section \ref{sec:fourier}. Periodic boundary conditions are assumed, unless otherwise stated. In the first part, we assess the performance of the method for the linear system \eqref{eq:coriolis}, thereby also validating the theoretical results in Section \ref{sec:fourier}. In a later part, we examine the application of the method to the case of nonlinear shallow water equations with Coriolis source terms.
\subsection{Linear system}
In this section, we focus on numerical experiments for the linear acoustics system with Coriolis source terms \eqref{eq:coriolis}. The reference speed $a_*$ in \eqref{eq:coriolis} is fixed at $a_* = 1.0$ for all experiments.\\
\begin{example} \label{ex:linvortex} (Stationary geostrophic vortex)
    In this example, from \cite{barsukow2025b}, we consider a stationary geostrophic equilibrium solution to \eqref{eq:coriolis}. The computational domain is chosen to be $[0,1] \times [0,1]$, and the Coriolis parameter in \eqref{eq:coriolis} is set to $c = 0.2$. The initial datum is given by
\begin{align}\label{ic:statvortex_linear}
    \boldsymbol{u}^{0}(x,y) &= -h (\|\boldsymbol{x} - \boldsymbol{x}_0 \|) \begin{pmatrix}
        x - x_0 \\
        y - y_0
    \end{pmatrix}^{\perp}, \\
    p^{0}(x,y) &= 1 - c g(\|\boldsymbol{x} - \boldsymbol{x}_0 \|), \nonumber
\end{align} 
where $\displaystyle g(\rho) := \frac{1}{10} \exp(-100\rho^2)$ and $\displaystyle h(\rho) := 20 \exp(-100\rho^2)$.
The initial datum \eqref{ic:statvortex_linear} (see Figure \ref{fig:combined_stat_vortex}(a)) is evolved  using the semi-discrete Active Flux method \eqref{eq:af_cellav_evo}, \eqref{eq:quas_pt} up to $t=10^3$  using a $20 \times 20$ grid (see Figure \ref{fig:combined_stat_vortex}(b)). Also, a diagonal cross-section (along $x=y$) of the numerical solution for $p$ (cell averages) at times $t=10, 100$ is plotted in Figure \ref{fig:combined_stat_vortex}(c). For comparison, we show the solution obtained by evolving the same initial datum \eqref{ic:statvortex_linear} until $t=100$ using the Active Flux method with Rusanov-type splitting \eqref{eq:rusanov} on a $40 \times 40 $ grid in  Figure \ref{fig:linvor_rus}. The degradation of the numerical solution by $t=100$ can be observed in this case.
\par
The Fourier transform of the initialization of the continuous initial datum \eqref{ic:statvortex_linear} need not strictly lie within the kernel \eqref{eq:af_kernel} of the Active Flux spatial discretization. \textcolor{black}{Instead, during an initial layer,} the numerical solution is observed to converge towards {the discrete steady state} as time evolves as displayed in the $L^1$-error plot in Figure \ref{fig:combined_stat_vortex}(d). Furthermore, the temporal evolution of the discrete geostrophic equilibrium residuals \eqref{eq:res1}-\eqref{eq:res11} which characterize the discrete stationary state for the Active Flux method, plotted in Figure \ref{fig:combined_stat_vortex}(e), confirms that the numerical solution approaches the exact discrete stationary state.  This is in contrast to a different discretization of the geostrophic equilibrium \eqref{eq:geo_balance}  \begin{align}
    c u^A_{i,j} + \frac{p^A_{i,j+1} - p^A_{i,j-1}}{2\Delta y} &= 0, &
    c v^A_{i,j} - \frac{p^A_{i+1,j} - p^A_{i-1,j}}{2\Delta x} &= 0, \label{eq:standard_div_x}
\end{align} also plotted in Figure \ref{fig:combined_stat_vortex}(e), which does not decay with time.
\begin{figure}[htbp]
    \centering

    \begin{subfigure}{\textwidth}
        \centering
        \includegraphics[width=\linewidth]{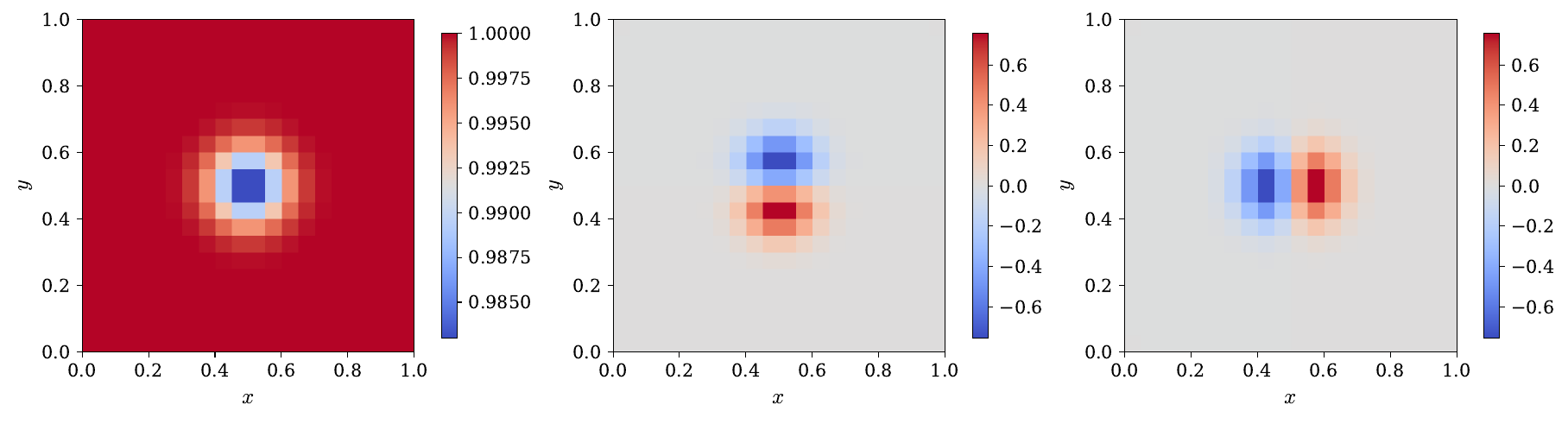}
        \caption{Initial datum ($p^0, u^0, v^0$) given by \eqref{ic:statvortex_linear}.}
        \label{fig:statvor_initial}
    \end{subfigure}

    \vspace{0.1cm}

    \begin{subfigure}{\textwidth}
        \centering
        \includegraphics[width=\linewidth]{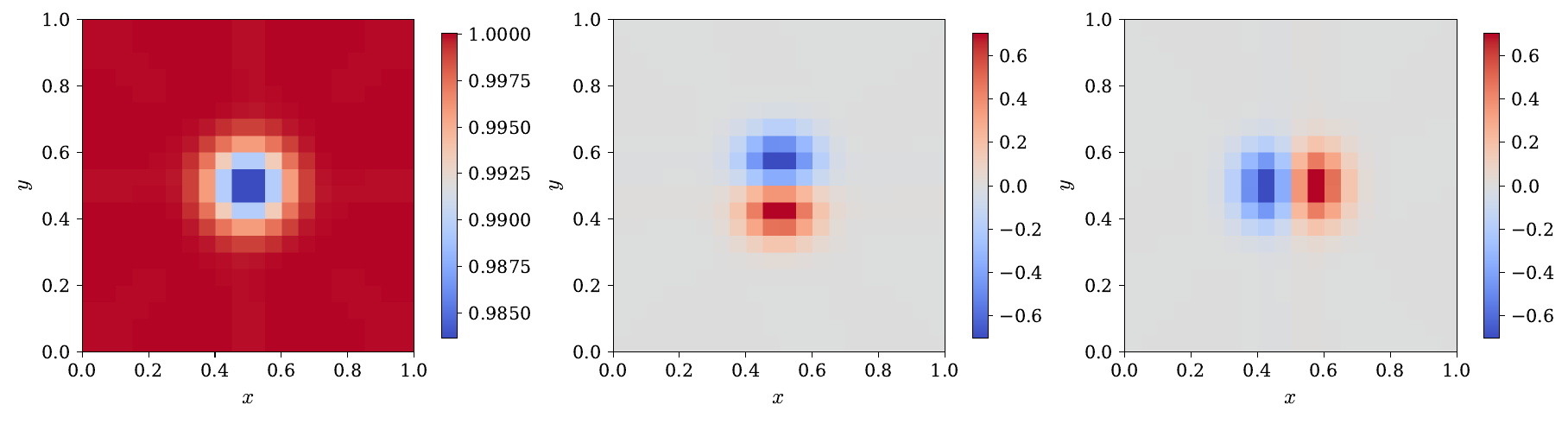}
        \caption{Active Flux solution of ($p, u, v$) (cell average) at $t=1000.0$.}
        \label{fig:statvor_t1000}
    \end{subfigure}

    \vspace{0.1cm}

    \begin{subfigure}[b]{0.45\textwidth}
        \centering
        \includegraphics[width=\textwidth]{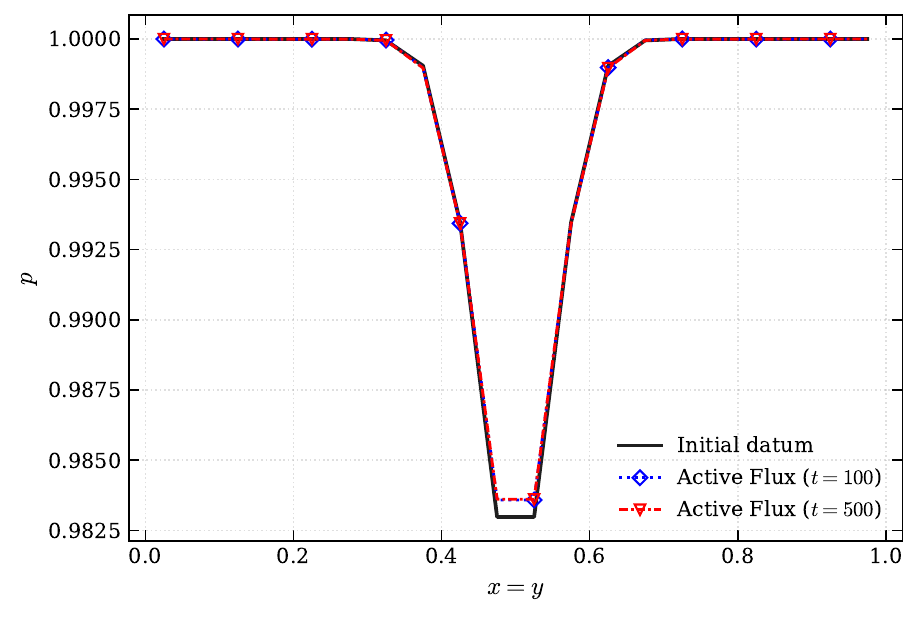}
        \caption{Diagonal cross-section ($x=y$) of $p.$}
        \label{fig:diagonal_global}
    \end{subfigure}
    \hfill
    \begin{subfigure}[b]{0.45\textwidth}
        \centering
        \includegraphics[width=\textwidth]{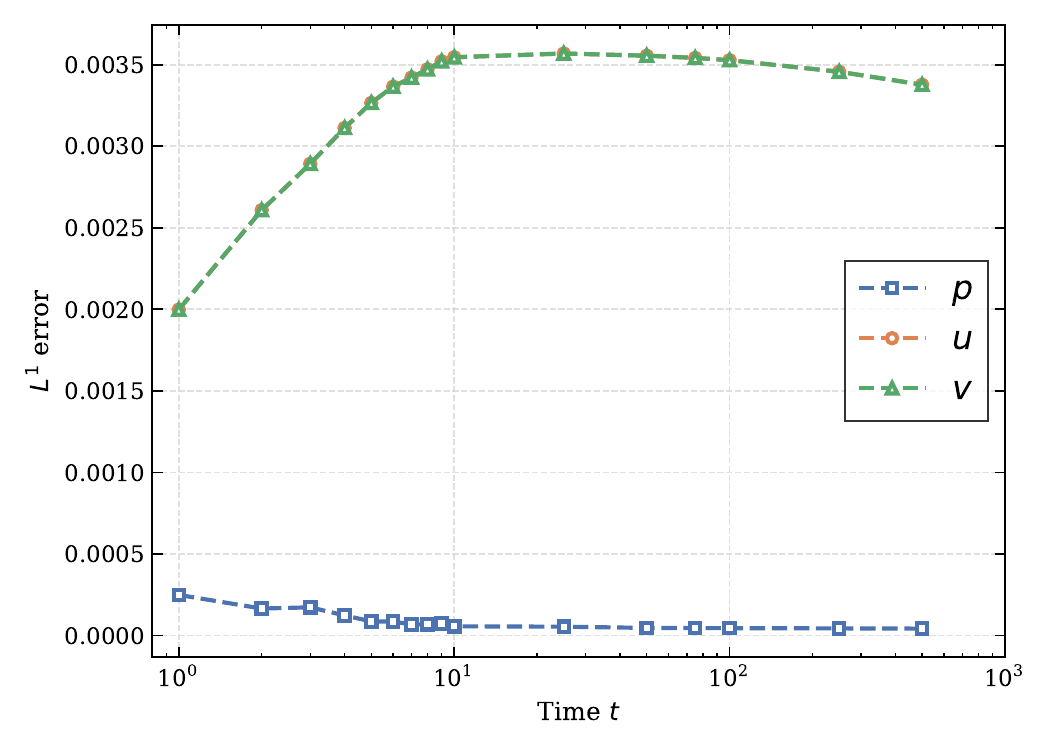}
        \caption{${L}^1$-error evolution until $t=500$.}
        \label{fig:statvor_error}
    \end{subfigure}

    \vspace{0.1cm}

    \begin{subfigure}[b]{0.45\textwidth}
        \centering
\includegraphics[width=\textwidth]{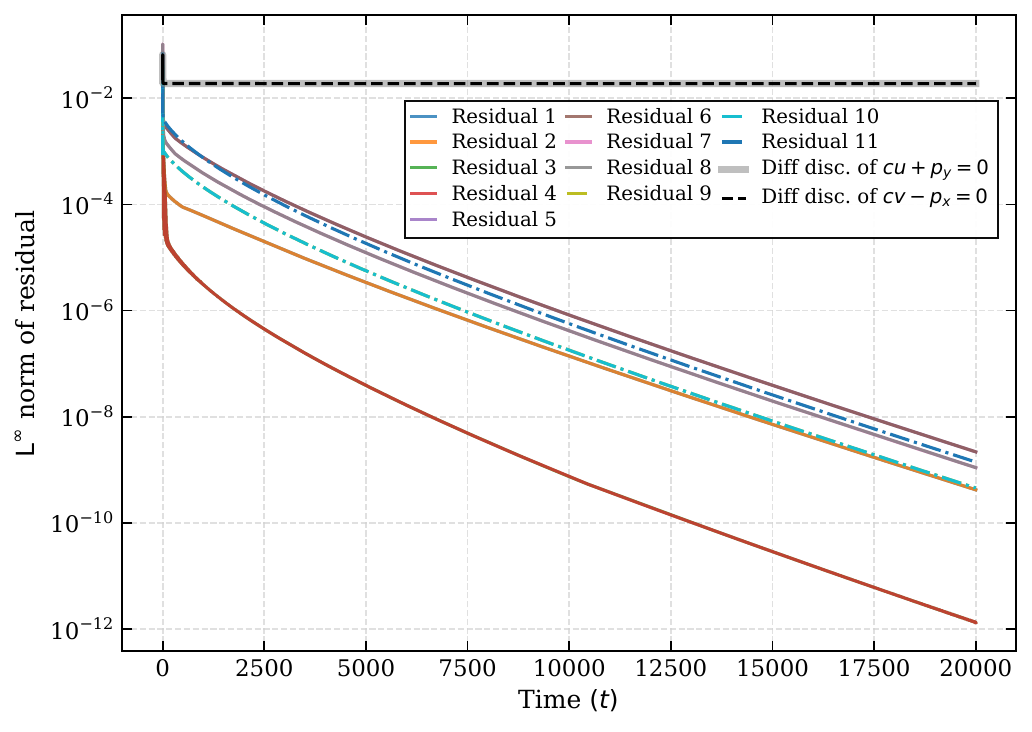}
        \caption{Decay of the discrete residuals \eqref{eq:res1}-\eqref{eq:res11}.}
        \label{fig:fddecay}
    \end{subfigure}
\hfill
      \begin{subfigure}{0.45\textwidth}
        \centering
        \includegraphics[width=\textwidth]{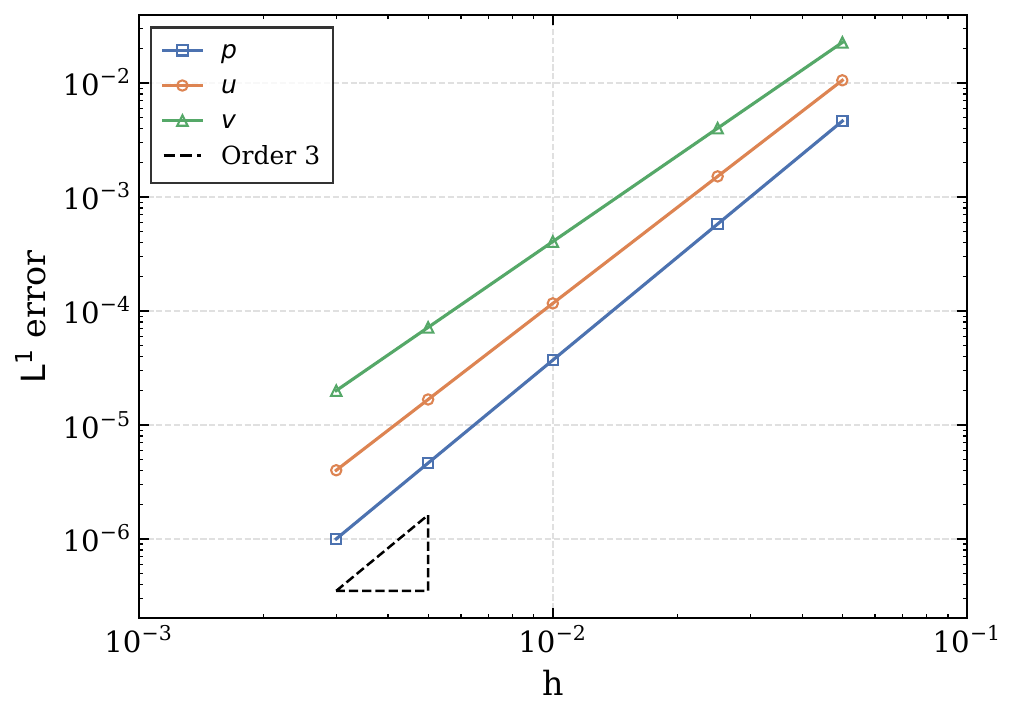}
        \caption{Log-log plot of $L^1$-errors in the cell-average solutions at $t = 0.1.$}
    \end{subfigure}

    \caption{Example \ref{ex:linvortex}. Active Flux evolution of the stationary vortex \eqref{ic:statvortex_linear} on a $20 \times 20$ grid.}
    \label{fig:combined_stat_vortex}
\end{figure}

\begin{figure}[htbp]
    \centering
    \begin{subfigure}{\textwidth}
        \centering
        \includegraphics[width=\linewidth]{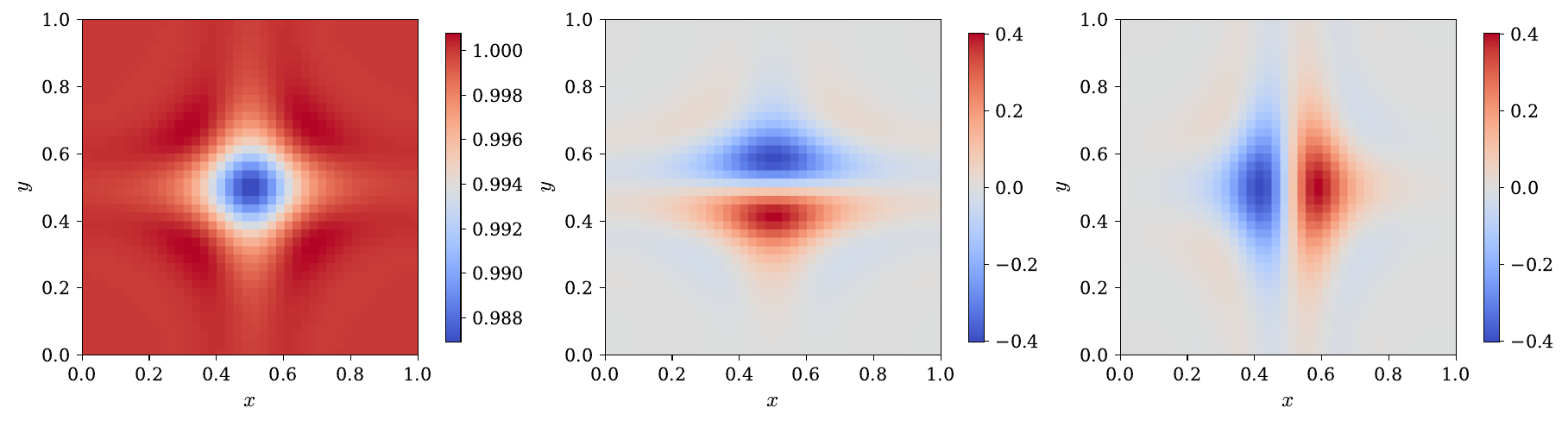}
    \end{subfigure}
    \caption{Example \ref{ex:linvortex}. Evolution of the stationary geostrophic vortex \eqref{ic:statvortex_linear} with a non-stationarity preserving method. Active Flux solution (cell-average) with Rusanov splitting \eqref{eq:rusanov} at $t=100$ on a $40 \times 40$ grid.}
    \label{fig:linvor_rus}
\end{figure}
 Next, we evaluate the experimental order of convergence (E.O.C.) of the method \eqref{eq:af_cellav_evo}, \eqref{eq:af_ptvalue}. The initial datum \eqref{ic:statvortex_linear} is evolved using the method up to time $t=1.0$ with different mesh sizes $\displaystyle h =\Delta x = \Delta y \in \left\{ \frac{1}{20},   \frac{1}{40}, \frac{1}{80}, \frac{1}{160} \right\}.$  The log-log plot of the ${L}^1$-errors against $\Delta x$ for the Active Flux method \eqref{eq:af_cellav_evo}, \eqref{eq:af_ptvalue} is plotted in Figure \ref{fig:combined_stat_vortex}(f). It is observed that the method attains the desired third-order convergence.

\end{example}

\begin{example}\label{ex: well_prep} (Well-prepared datum) In this example, we construct a discrete initial datum which is \emph{well-prepared} in the sense that its Fourier transform belongs to the span of the vector \eqref{eq:af_kernel}. In particular, we seek a well-prepared solution of the form $\hat{\boldsymbol{q}} \exp(\mathbb{i}k_{x}i\Delta x + \mathbb{i}k_{y}j\Delta y),$ where the vector $\hat{\boldsymbol{q}}$ is parallel to \eqref{eq:af_kernel}. This would lead to a complex-valued solution, so we consider only the real parts, which can also be understood as the average of the solutions corresponding to the wavenumbers $\boldsymbol{k} = (k_x, k_y)$ and  $-\boldsymbol{k} = (-k_x, -k_y).$ 
\par
First, we fix the wave numbers $k_x = 2 \pi, k_y = 20 \pi,$  $\theta_x = k_x \Delta x, 
    \theta_y = k_y \Delta y.$
Let the index $(i,j)$ correspond to the physical coordinate $(i\Delta x, j\Delta y)$. Denote $\Phi^{N}_{i,j} = i \theta_x + j \theta_y,  \Phi^{E_H}_{i,j} = \left(i - \frac{1}{2}\right) \theta_x + j \theta_y, \Phi^{E_V}_{i,j} = i \theta_x + \left(j - \frac{1}{2}\right) \theta_y, \Phi^{A}_{i,j} = \left(i - \frac{1}{2}\right) \theta_x + \left(j - \frac{1}{2}\right) \theta_y.
$ We start by setting $\hat{p}^N_{i,j} = 1,$ from which we obtain $p^{N}_{i,j} = \cos(\Phi^{N}_{i,j}) = \textrm{Re}(\exp(\mathbb{i}(i \theta_x + j \theta_y))).$ Using further algebraic simplifications, the remaining well-prepared discrete degrees of freedom can be derived as follows:
\begin{align}\label{eq:wellprep_data}
    u^{N}_{i,j} &= \frac{2 \tan(\theta_y/2)}{c \Delta y} \sin(\Phi^{N}_{i,j}), \quad 
    v^{N}_{i,j} = -\frac{2 \tan(\theta_x/2)}{c \Delta x} \sin(\Phi^{N}_{i,j}),\\  \quad 
    p^{E_H}_{i,j} & = \frac{3+\cos\theta_x}{4 \cos(\theta_x/2)} \cos(\Phi^{E_H}_{i,j}), \quad 
    u^{E_H}_{i,j} = \frac{(3+\cos\theta_x)\tan(\theta_y/2)}{2 c \Delta y \cos(\theta_x/2)} \sin(\Phi^{E_H}_{i,j}), \\[6pt] 
    v^{E_H}_{i,j} &= -\frac{2\sin(\theta_x/2)}{c \Delta x} \sin(\Phi^{E_H}_{i,j}) , \quad 
    p^{E_V}_{i,j} = \frac{3+\cos\theta_y}{4 \cos(\theta_y/2)} \cos(\Phi^{E_V}_{i,j}),\\[12pt] 
    u^{E_V}_{i,j} &= \frac{2\sin(\theta_y/2)}{c \Delta y} \sin(\Phi^{E_V}_{i,j}), \quad 
    v^{E_V}_{i,j} = -\frac{(3+\cos\theta_y)\tan(\theta_x/2)}{2 c \Delta x \cos(\theta_y/2)} \sin(\Phi^{E_V}_{i,j}), \\[6pt] 
    p^{A}_{i,j} &= \frac{(2+\cos\theta_x)(2+\cos\theta_y)}{9 \cos(\theta_x/2) \cos(\theta_y/2)} \cos(\Phi^{A}_{i,j}),\\[6pt]
    u^{A}_{i,j} &= \frac{2(2+\cos\theta_x)\sin(\theta_y/2)}{3 c \Delta y \cos(\theta_x/2)} \sin(\Phi^{A}_{i,j}), \\[6pt]
    v^{A}_{i,j} &= -\frac{2(2+\cos\theta_y)\sin(\theta_x/2)}{3 c \Delta x \cos(\theta_y/2)} \sin(\Phi^{A}_{i,j}).
\end{align}

The $L^1$-errors (with the well-prepared initial datum as the reference) in the Active Flux numerical solution $(u,v,p)$ obtained upon evolving the well-prepared discrete initial datum \eqref{eq:wellprep_data} until  $t=10^3$ is displayed in Figure \ref{fig:af_wellprep}. The $L^1$-errors  are close to the machine precision and indicate only the presence of round-off errors.
\begin{figure}[htbp]
        \centering
        \includegraphics[width=0.5\textwidth]{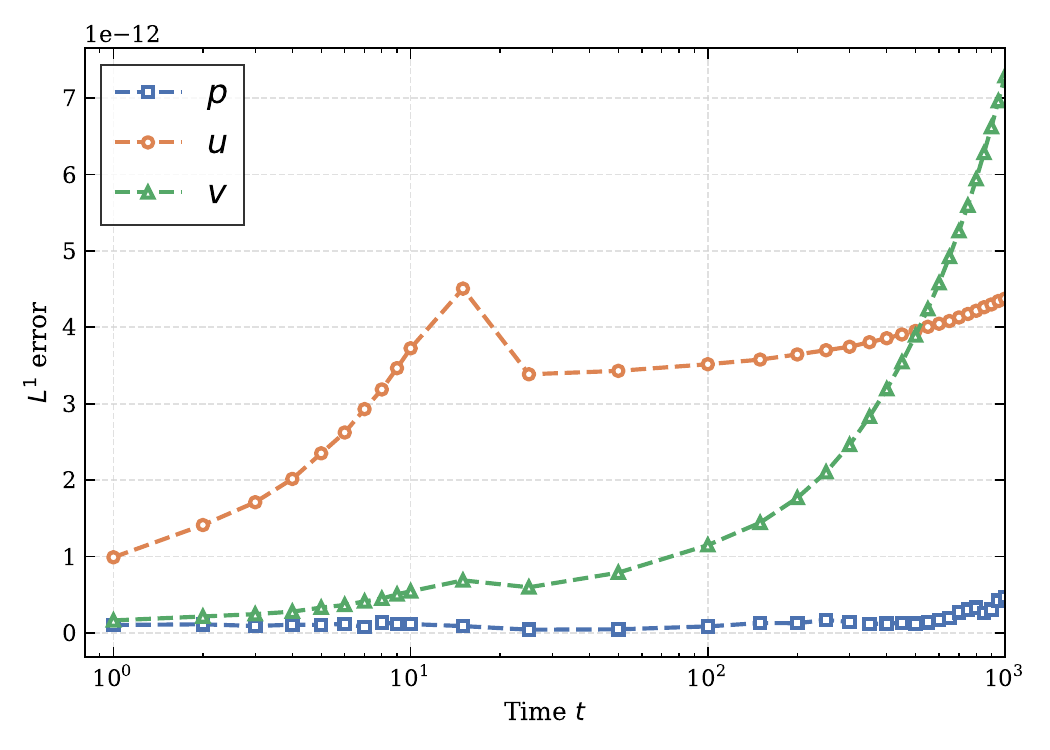}
\caption{Example \ref{ex: well_prep}. ${L}^1$-errors in the Active Flux (cell-average) solutions on a $50 \times 50$ grid as a function of time, upon evolving the well-prepared initial datum.}
    \label{fig:af_wellprep}
\end{figure}
\par
Next,  a local perturbation is added to the pressure ($p$) component of the well-prepared discrete stationary state \eqref{eq:wellprep_data}. The perturbed initial datum for the variable $p$ is given by 
\begin{align}\label{eq:perturb_data}
p(x,y,0) &= p_{\text{wp}}(x,y) + \delta p(x,y), & 
    \delta p(x,y) &= 
    \begin{cases} 
        \theta \exp\left( 1 - \frac{1}{1 - (r/r_0)^2} \right) & \text{if } r < r_0, \\
        0 & \text{otherwise},
    \end{cases}
\end{align}
with $r = \sqrt{(x-x_p)^2 + (y-y_p)^2},$ $(x_p, y_p) = (0.4, 0.43)$ with a radius of {$r_0 = 0.02$}  and an amplitude of $\theta = 10^{-2}$. The initial velocity datum is not perturbed.  Plot of $\norm{\boldsymbol{q}^{\mbox{num}}_{ij} - \boldsymbol{q}^{WP}_{ij}}_{2} = \sqrt{(u^{\mbox{num}}_{ij} - u^{WP}_{ij})^{2}+ (v^{\mbox{num}}_{ij} - v^{WP}_{ij})^{2}+ (p^{\mbox{num}}_{ij} - p^{WP}_{ij})^{2}}$ for the evolved semi-discrete Active Flux numerical solution (cell-average) $\boldsymbol{q}^{\mbox{num}}$ at $t=0.3$ is displayed in Figure \ref{fig:af_perturbation_error_prop_a}(a). We observe that the perturbation is propagated effectively by the semi-discrete Active Flux method without any spurious waves (The result is comparable to the result in Figure 4 of \cite{barsukow2025b}), thereby also validating its stationarity preservation property. This is in contrast to the case of the Rusanov splitting method \eqref{eq:rusanov}, which is not effective in propagating the perturbation, as displayed in Figure \ref{fig:af_perturbation_error_prop_a}(b). 
\begin{figure}[htbp]
    \centering

    \begin{subfigure}[b]{0.48\textwidth}
        \centering
        \includegraphics[width=\linewidth]{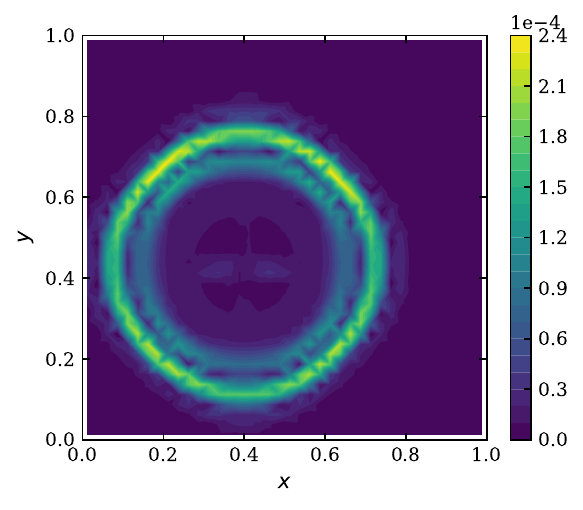}
        \caption{Active Flux}
        \label{fig:perturb_af}
    \end{subfigure}
    \hfill
    \begin{subfigure}[b]{0.48\textwidth}
        \centering
        \includegraphics[width=\linewidth]{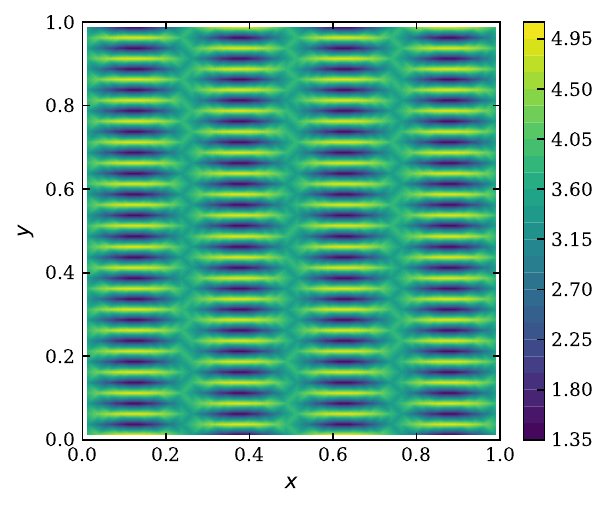}
        \caption{Active Flux with Rusanov splitting}
        \label{fig:perturb_rus}
    \end{subfigure}

    \caption{Example \ref{ex: well_prep}. Evolution of a perturbation of the well-prepared (WP) solution. Plot of $\norm{\boldsymbol{q}_{\mbox{num}} - \boldsymbol{q}_{WP}}_{2}$ for the
    cell-average solution at $t=0.3$ for the perturbed initial
    datum \eqref{eq:perturb_data} on a $40\times40$ grid with
    $\delta=10^{-2}$.}
    \label{fig:af_perturbation_error_prop_a}
\end{figure}
\end{example}
\subsection{Nonlinear shallow water equations}
In this section, we consider the nonlinear shallow water equations on a flat bottom topography with Coriolis source terms \eqref{eq:swe_noncons}, written in the conservative form \eqref{eq:coriolis}
where 
\begin{align}\label{eq:swesys_vars}
   \boldsymbol{q} &= \begin{pmatrix} h \\ hu \\ hv \end{pmatrix}, \,\,  
\boldsymbol{f}^x =  \begin{pmatrix} hu \\ \frac{(hu)^2}{h} + \frac{1}{2}gh^2 \\ \frac{huhv}{h} \end{pmatrix}, \,\, \boldsymbol{f}^y =  
    \begin{pmatrix} hv \\ \frac{huhv}{h} \\ \frac{(hv)^2}{h} + \frac{1}{2}gh^2 \end{pmatrix}, \,\, \boldsymbol{s}(\boldsymbol{q}) =  
    \begin{pmatrix} 0 \\ \Omega hv \\ -\Omega hu, \end{pmatrix},
\end{align}
where $h$ is the water depth, $u$ and $v$ are the velocities, $g$ is the gravitational acceleration, and $\Omega$ is the Coriolis parameter. We only consider the system \eqref{eq:coriolis}-\eqref{eq:swesys_vars} with a constant $\Omega$ as in  \cite{audusse2018, audusse2025}.\\
\textbf{Active Flux for the nonlinear system.} 
Since the Coriolis source term in the nonlinear shallow water system \eqref{eq:coriolis}-\eqref{eq:swesys_vars} is a linear function of the conservative momentum variables $hu$ and $hv,$ the Active Flux numerical method can be formulated in a way similar to the linear case \eqref{eq:af_cellav_evo}, \eqref{eq:af_ptvalue}. The crucial difference is that point value evolution formula  \eqref{eq:af_ptvalue} now relies on the state-dependent Jacobian matrices $J_x(\boldsymbol{q}) = \partial f^x / \partial \boldsymbol{q}$ and $J_y(\boldsymbol{q}) = \partial f^y / \partial \boldsymbol{q}$ which are defined using the primitive velocities $u = hu/h$ and $v = hv/h$ as follows
\begin{equation*}
    J_x(\boldsymbol{q}) = \begin{pmatrix} 0 & 1 & 0 \\ -u^2 + gh & 2u & 0 \\ -uv & v & u \end{pmatrix}, \quad 
    J_y(\boldsymbol{q}) = \begin{pmatrix} 0 & 0 & 1 \\ -uv & v & u \\ -v^2 + gh & 0 & 2v \end{pmatrix}.
\end{equation*}
The positive and negative split Jacobians ($J_x^\pm$ and $J_y^\pm$) required in \eqref{eq:af_ptvalue} are then computed pointwise using the eigenvalue decomposition $J(\boldsymbol{q}^P) = R(\boldsymbol{q}^P) \Lambda(\boldsymbol{q}^P) R^{-1}(\boldsymbol{q}^P).$ For the $x$-direction, the eigenvalues are $\lambda_x \in \{u-c, u, u+c\}$, and for the $y$-direction, $\lambda_y \in \{v-c, v, v+c\}$, where $c = \sqrt{gh}$.

\begin{example}(Stationary vortex) \label{ex:swesys}
We consider a stationary vortex test case for the nonlinear system \eqref{eq:coriolis}-\eqref{eq:swesys_vars}, considered in \cite{audusse2018}.
The computational domain is $[-0.5, 0.5]^2.$ The initial velocity profile is defined as $\boldsymbol{u}^{0 } = u_\theta(r) \bar{\boldsymbol{e}}_{\theta},$ where for $\epsilon > 0,$ we consider
\begin{equation}\label{eq:swe_test_velocity}
    u_\theta(r) = \varepsilon \left[ 5r \chi_{\left\{r < \frac{1}{5}\right\}} + (2 - 5r) \chi_{\left\{\frac{1}{5} \le r < \frac{2}{5}\right\}} \right],
\end{equation}
where $\bar{\boldsymbol{e}}_{\theta} = (-\sin(\theta), \cos(\theta)),$ $r = \sqrt{x^2 + y^2}$ is the radial distance from the origin  and $\varepsilon > 0$ is the vortex strength parameter. We set  $g=9.81$, $\Omega=1$, and the initial water height $(h^0)$ is obtained as the $C^{0}$ solution to 
\begin{equation}\label{eq:swe_test_height}
    \partial_r h^0 = \frac{1}{g} \left( \Omega u_\theta + \frac{u_\theta^2}{r} \right) = \begin{cases}
        \frac{1}{g}(5\epsilon \Omega r + 25\epsilon^{2}r), \quad r < \frac{1}{5},\\
          \frac{1}{g}(\epsilon (2-5r) \Omega  + \epsilon^{2}\frac{(2-5r)^2}{r}, \quad \frac{1}{5} \leq r < \frac{2}{5},\\
          0, \quad \mbox{elsewhere.}
    \end{cases}
\end{equation}This is an exact stationary solution of the nonlinear system \eqref{eq:coriolis}-\eqref{eq:swesys_vars},  but it is not in geostrophic equilibrium \eqref{eq:geo_eq} since $g\nabla h^0  +{\Omega}\left(\boldsymbol{u}^0\right)^\perp = \frac{1}{r}(u_\theta)^2\boldsymbol{e}_{r},$ where $ \boldsymbol{e}_{r} = (\cos(\theta),\sin(\theta))$ is the radial unit vector. \textcolor{black}{However, when $\epsilon \rightarrow 0,$  the stationary solution \eqref{eq:swe_test_velocity}-\eqref{eq:swe_test_height} becomes close to the geostrophic equilibrium, i.e., $ g\nabla h^0  +{\Omega}\left(\boldsymbol{u}^0\right)^\perp  \approx 0$, see \cite{audusse2018, audusse2025}. We note that scaling the velocity field with a parameter $\epsilon \rightarrow 0,$  corresponds to low Froude and low Rossby numbers (which are proportional to the velocity field). The numerical solution (cell-average) obtained by evolving the initial datum \eqref{eq:swe_test_velocity}-\eqref{eq:swe_test_height} with $\epsilon = 0.01$ on a $40 \times 40$ grid up to time $t=200$ is shown in Figure \ref{fig:swe_combined_vortex}. 
In  Figure \ref{fig:rel_err}, following \cite{audusse2025}, for $\epsilon = 10^{-1}, 10^{-2}, 10^{-3} \mbox{and}\, 10^{-4},$  we plot the relative error of the water height at $t=200,$ given by $\displaystyle{|\min(h_{final}) - \min(h_{init})|}/{\max(h_{init}) - \min(h_{init})}.$ The relative error is observed to decay as $\epsilon \rightarrow 0,$ indicating that the method provides reliable approximations to the stationary vortex in low Froude-low Rossby regimes. In contrast, Figure \ref{fig:stat_soln_rusanov_swe} shows that evolving the datum \eqref{eq:swe_test_velocity}-\eqref{eq:swe_test_height} with $\epsilon=0.001$ up to t=500 using an Active Flux method with Rusanov flux splitting \eqref{eq:rusanov} leads to a breakdown of the vortex structure.}
   \begin{figure}[htbp]
    \centering

    \begin{subfigure}{\textwidth}
        \centering
        \includegraphics[width=0.9\linewidth, height=0.33\textheight, keepaspectratio]{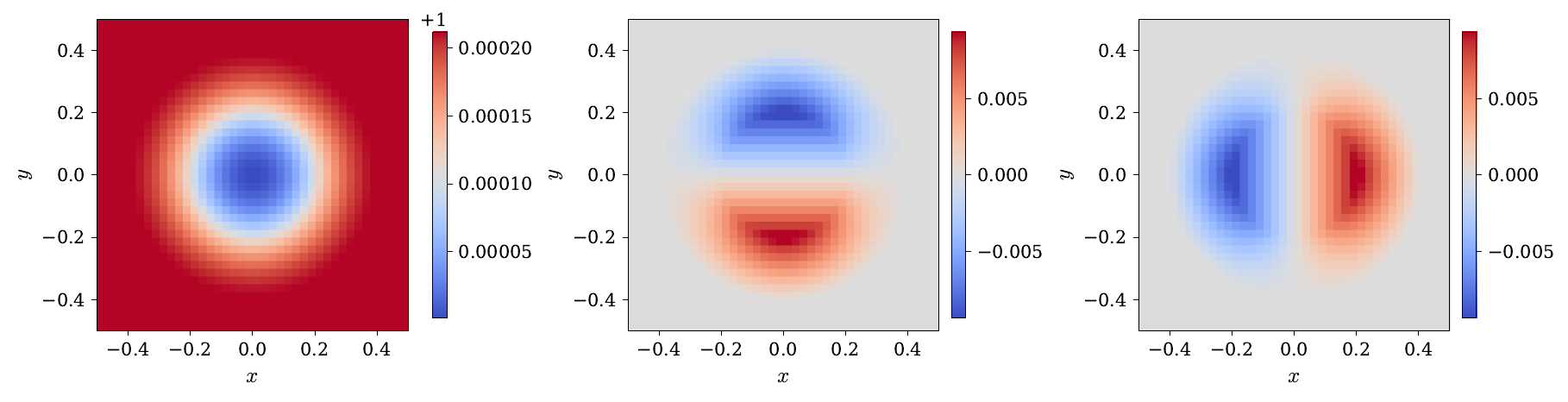}
        \caption{Initial datum ($h^0, (hu)^{0}, (hv)^{0}$) given by \eqref{eq:swe_test_velocity}-\eqref{eq:swe_test_height} with $\epsilon = 0.01.$}
        \label{fig:swe_initial}
    \end{subfigure}

    \vspace{0.1cm}

    \begin{subfigure}{\textwidth}
        \centering
        \includegraphics[width=0.9\linewidth, height=0.33\textheight, keepaspectratio]{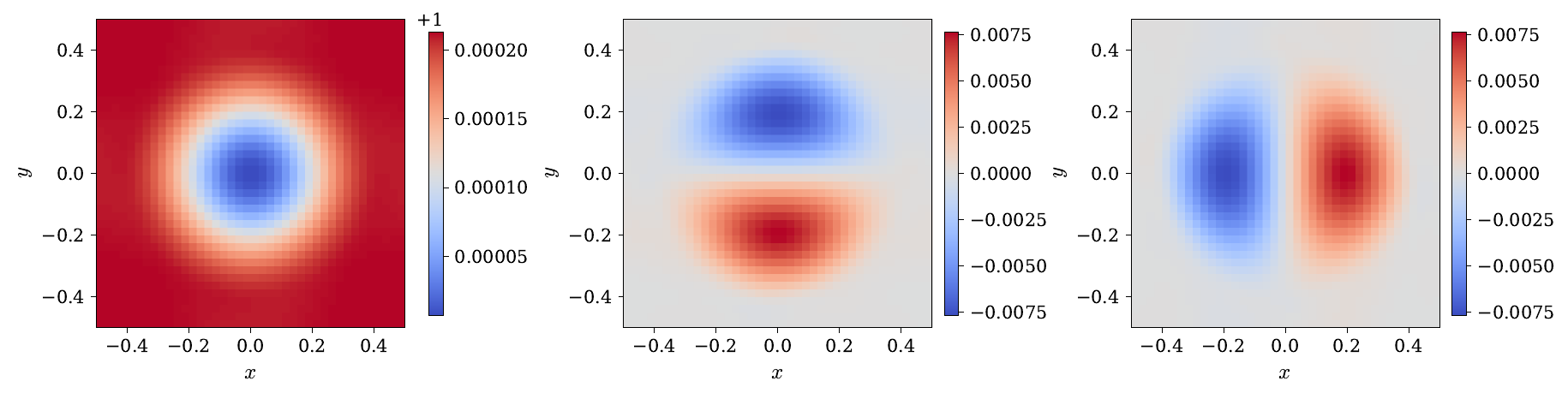}
        \caption{Active Flux solution (cell-average) at $t=200.0$ on a $40 \times 40$ grid.}
        \label{fig:swe_t200}
    \end{subfigure}

    \vspace{0.1cm}

    \begin{subfigure}[b]{0.47\textwidth}
        \centering
        \includegraphics[width=\textwidth]{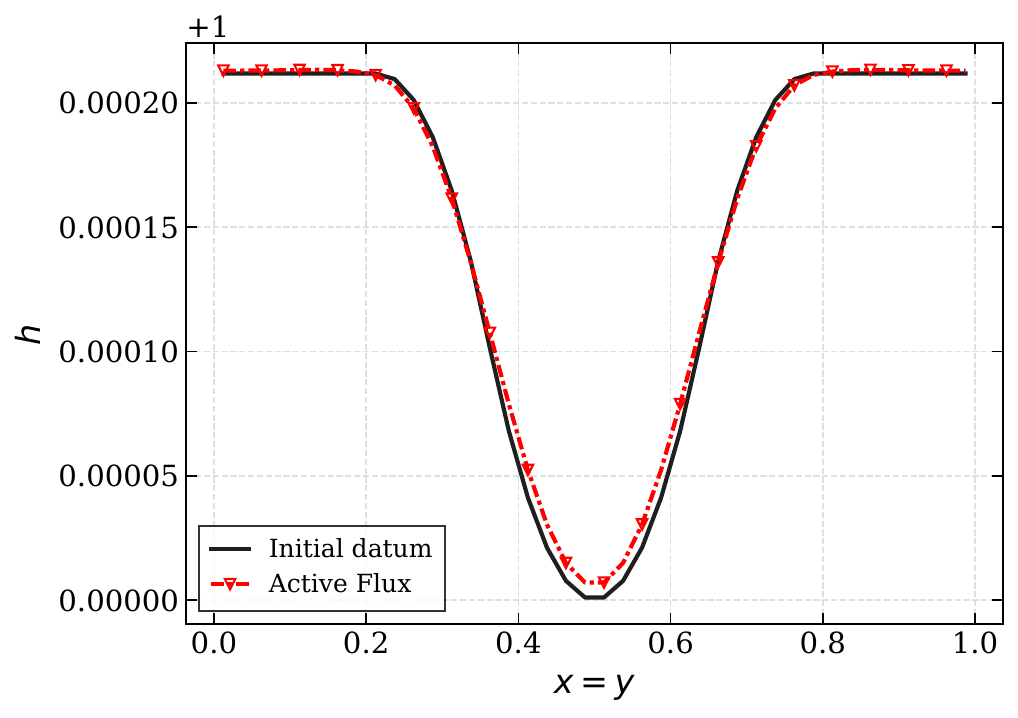}
        \caption{Diagonal 1D cut of $h$ (along $x=y$).}
        \label{fig:swe_diag_h}
    \end{subfigure}
    \hfill
    \begin{subfigure}[b]{0.47\textwidth}
        \centering
        \includegraphics[width=\textwidth]{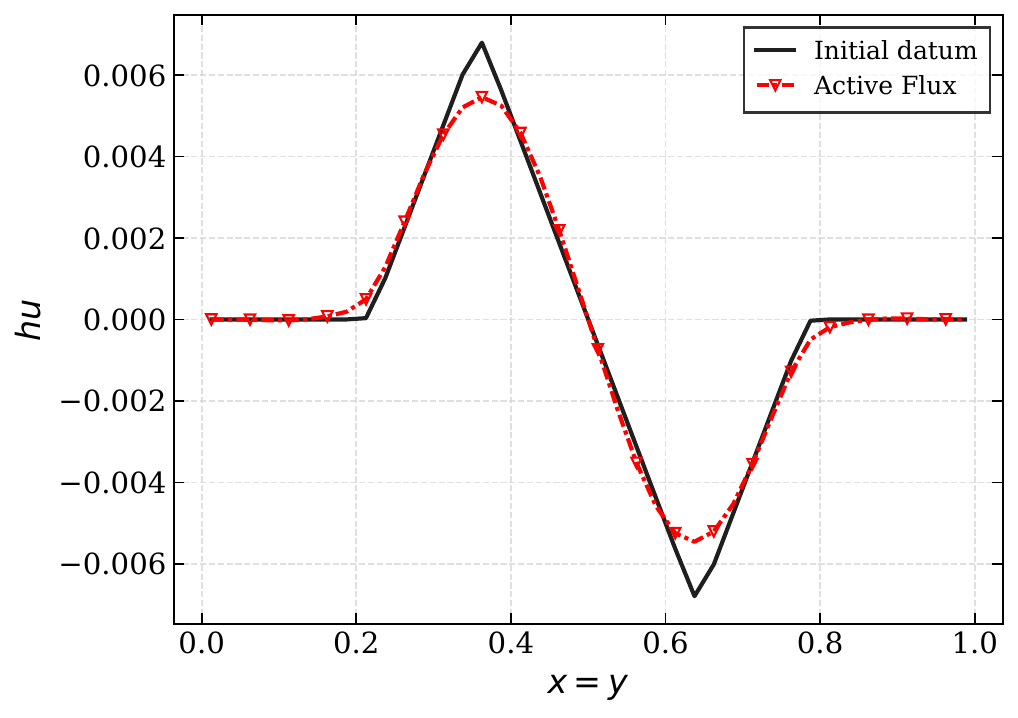}
        \caption{Diagonal 1D cut of $hu$ (along $x=y$).}
        \label{fig:swe_diag_hu}
    \end{subfigure}

    \caption{Example \ref{ex:swesys}. Active Flux solution  on a $40\times40$ grid at time $t=200$ obtained by evolving the stationary vortex \eqref{eq:swe_test_velocity}-\eqref{eq:swe_test_height} where $\epsilon=0.01$. The bottom row displays the diagonal ($x=y$) 1D cuts of the solution.}
    \label{fig:swe_combined_vortex}
\end{figure}

\begin{figure}[htbp]
    \centering

        \centering
        \includegraphics[width=0.5\linewidth, height=0.38\textheight, keepaspectratio]{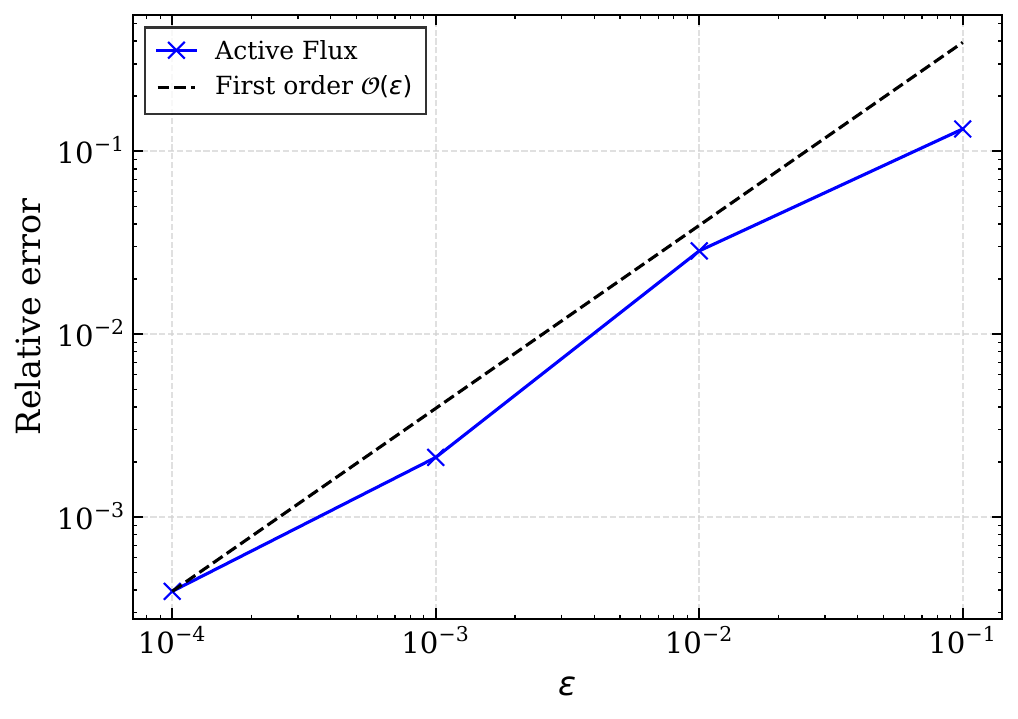}
    \caption{Example \ref{ex:swesys}. Relative error of water depth at $t=200$ obtained with the Active Flux method as the Froude parameter $\epsilon$ in \eqref{eq:swe_test_velocity}-\eqref{eq:swe_test_height} tends to 0.}
    \label{fig:rel_err}
\end{figure}

\begin{figure}[htbp]
    \centering
    \begin{subfigure}{\textwidth}
        \centering
        \includegraphics[width=\linewidth]{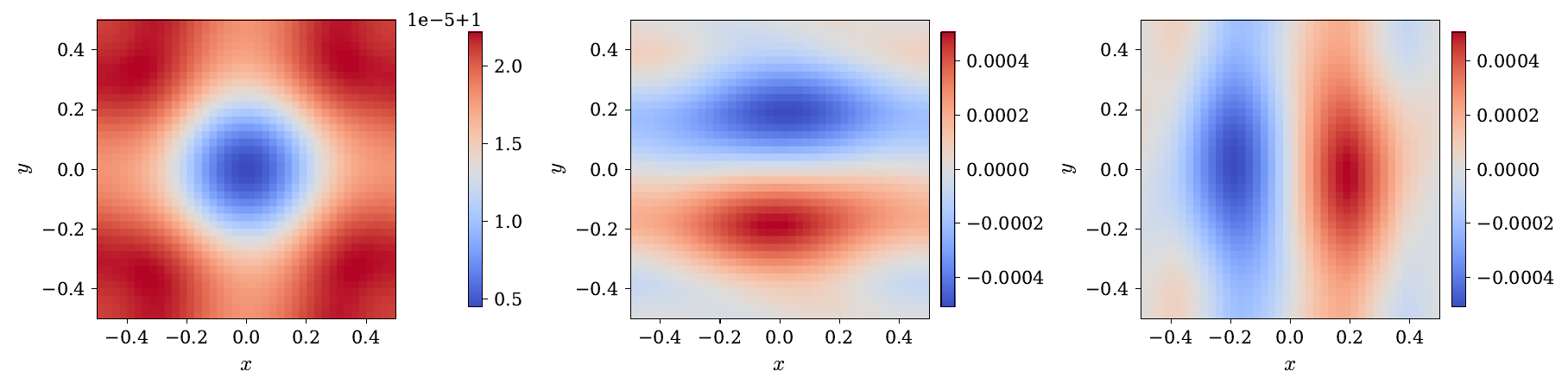}
    \end{subfigure}
    \caption{Example \ref{ex:swesys}. The evolution of the stationary vortex  \eqref{eq:swe_test_velocity}-\eqref{eq:swe_test_height} with $\epsilon =  0.001$ using Active Flux with Rusanov splitting \eqref{eq:rusanov}. The degradation of the vortex structure (cell-average solution of $(h, hu, hv)$ is shown) by time $t=500,$ when computed on a $40 \times 40$ grid can be observed.}
    \label{fig:stat_soln_rusanov_swe}
\end{figure}
\end{example}

\begin{example}(Translated vortex)\label{ex:tra_vor}
Next, we simulate the translated vortex test case studied in \cite{audusse2025}. For $\epsilon >0,$ the initial datum is given by:
\begin{align}
\label{ic:travor}
h^0(x,y) &= 1 + \frac{\epsilon}{g}x + 
\begin{cases} 
\frac{5\Omega\epsilon}{2g}r^2, & r \le 0.2,\\ 
\frac{\Omega\epsilon}{10g} - \frac{\Omega\epsilon}{g}(0.3 - 2r + 2.5r^2), & 0.2 < r \le 0.4, \\ 
\frac{\Omega\epsilon}{5g}, & r > 0.4, 
\end{cases}\\
\boldsymbol{u}^0(x,y) &= \begin{pmatrix} 0 \label{ic:travor_vel}\\ \frac{\epsilon}{\Omega} \end{pmatrix} + 
\begin{cases} 
-5\epsilon r\boldsymbol{e}_{\theta}, & r \le 0.2,\\ 
-(2-5r)\epsilon \boldsymbol{e}_{\theta}, & 0.2 < r \le 0.4,\\ 
\boldsymbol{0}, & r > 0.4, 
\end{cases}
\end{align}
where $r = \sqrt{x^2+y^2},$  $\theta$ are the polar coordinates and $\boldsymbol{e}_{\theta}:= (\sin(\theta), \cos(\theta))^T$. We set the parameters to $g=1$, $\Omega=1$, and $\epsilon=0.01$. The computational domain is $[-0.5, 0.5]^2,$ which is discretized using a $40 \times 40$ grid. \textcolor{black}{We note that the datum \eqref{ic:travor} is in geostrophic equilibrium.} Similarly to \cite{audusse2025}, we enforce periodic boundary conditions at $y = \pm 0.5$, and wall-type boundary conditions at $x = \pm 0.5.$ Figure~\ref{fig:translated_vortex} displays the initial datum and the Active Flux cell-average solutions $(h, hu, hv)$ at $t=20.0$. The Active Flux method is observed to preserve the vortex structure and amplitude of the solution.
\begin{figure}[htbp]
    \centering
    
    \begin{subfigure}{\textwidth}
        \centering
        \includegraphics[width=0.86\linewidth, height=0.38\textheight, keepaspectratio]{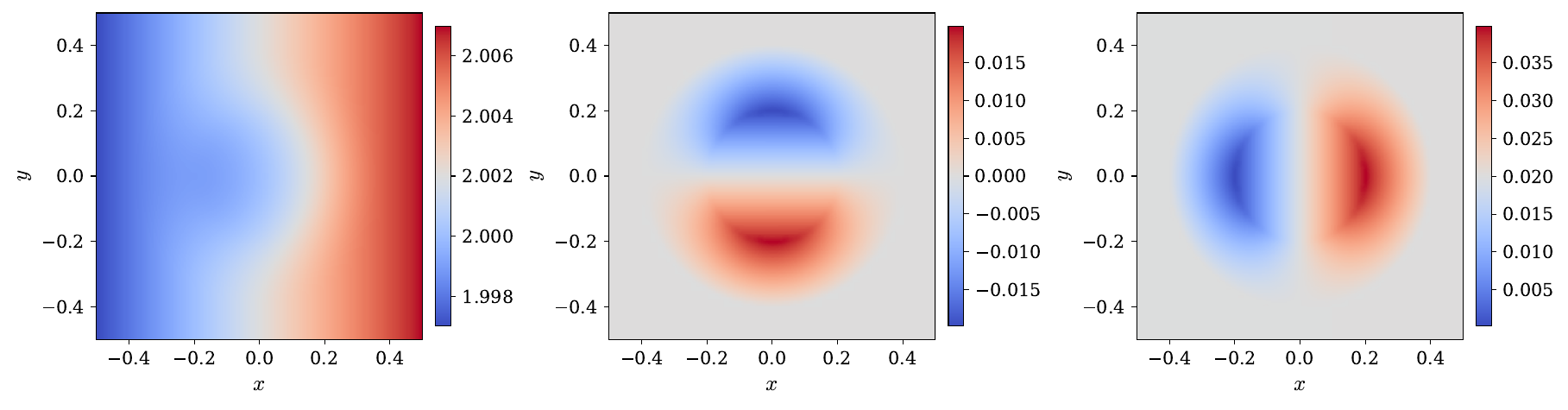}
        \caption{Initial datum  ($h^0, (hu)^{0}, (hv)^{0}$) given by \eqref{ic:travor}-\eqref{ic:travor_vel}.}
    \end{subfigure}

    \vspace{0.1cm}

    \begin{subfigure}{\textwidth}
        \centering
        \includegraphics[width=0.86\linewidth, height=0.38\textheight, keepaspectratio]{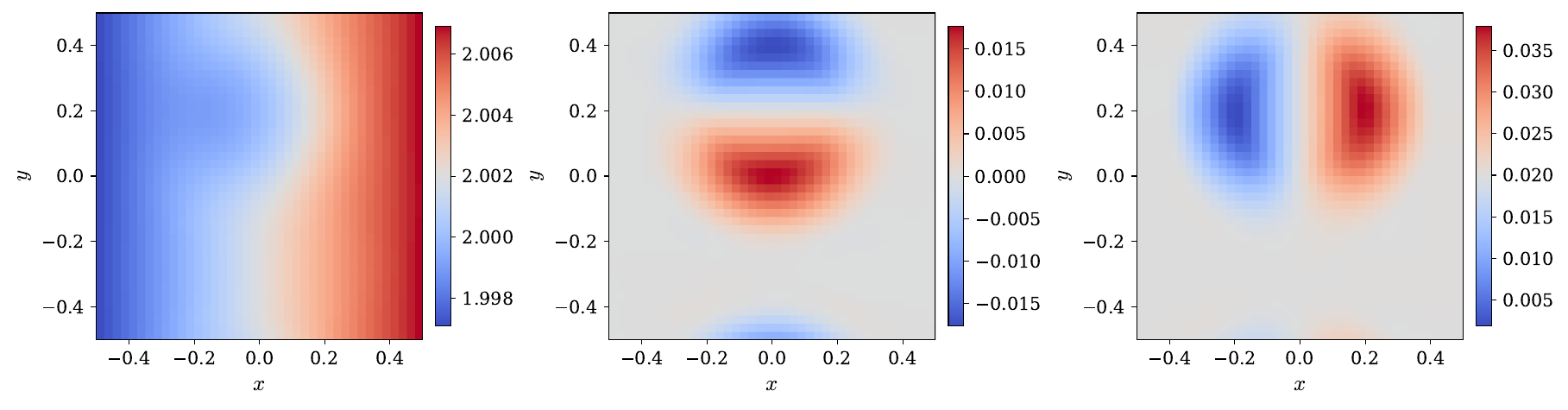}
        \caption{Active Flux solution  $(h, hu, hv)$ at $t=20.$}
    \end{subfigure}
    \vspace{0.1cm}
    \begin{subfigure}{0.48\textwidth}
        \centering
        \includegraphics[width=\linewidth]{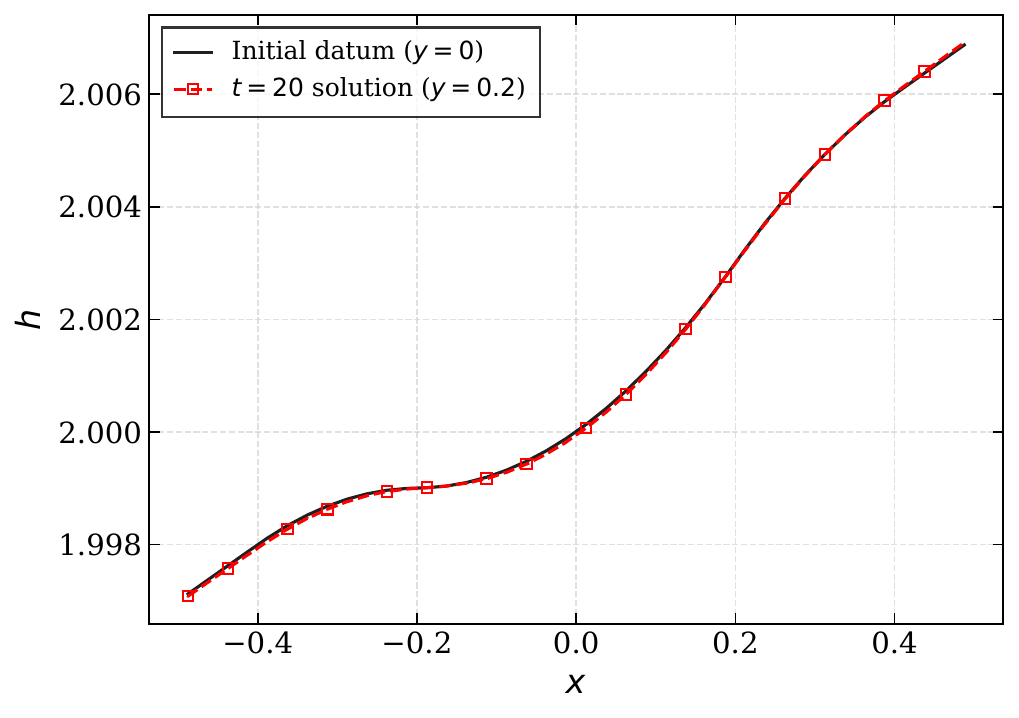}
        \caption{Cross section ($y=0$) of $h$ for the initial datum and $y=0.2$ for the Active Flux solution (cell-averages) at $t=20.0$.}
    \end{subfigure}
    \hfill
    \begin{subfigure}{0.48\textwidth}
        \centering
        \includegraphics[width=\linewidth]{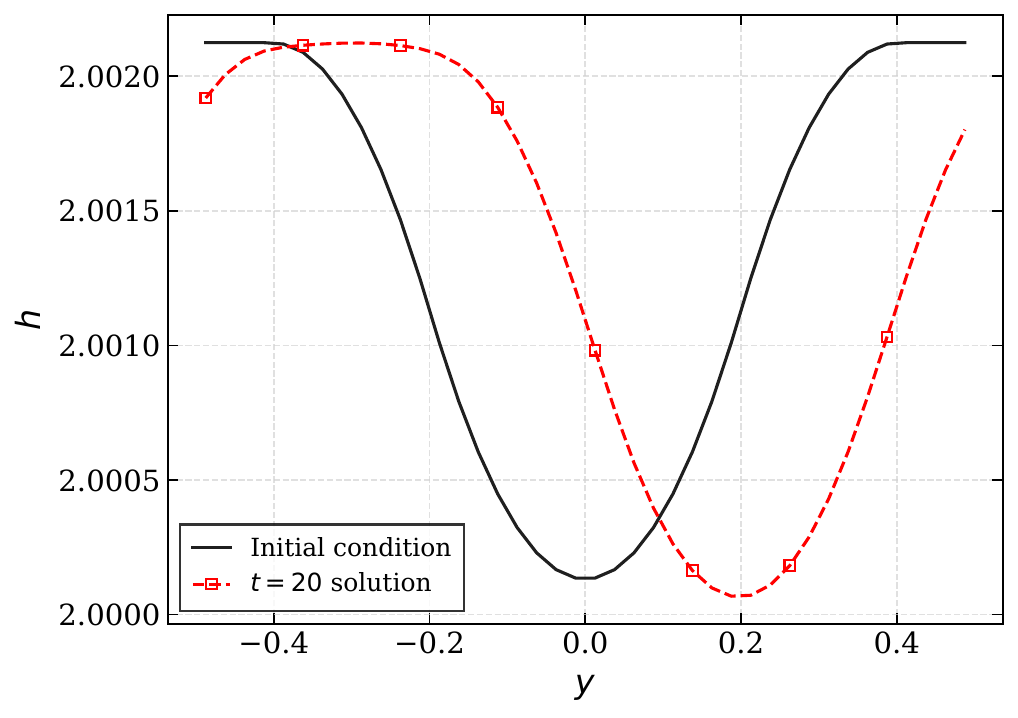}
        \caption{Cross section ($x=0$) of $h$ for the initial datum and the Active Flux solution (cell-averages) at $t=20.0$.}
    \end{subfigure}

    \caption{Example \ref{ex:tra_vor}. Active Flux solution on a $40 \times 40$ grid with $\epsilon = 0.01$.}
    \label{fig:translated_vortex}
\end{figure}
\end{example}

\begin{example}{(Geostrophic adjustment)}\label{ex:geo_adj} Finally, we consider the geostrophic adjustment test proposed in \cite{castro2008}. The computational domain is $[-10, 10]^2.$ We set the gravity and Coriolis parameters to $g=1$ and $\omega=1$. The fluid is initially at rest with an elliptical perturbation in the water height:
\begin{align}\label{ic:geoadj}
    h_0(x, y) &= 1 + \frac{1}{4} \left( 1 - \tanh\left(10 \left( \sqrt{2.5x^2 + 0.4y^2} - 1 \right)\right) \right), \\
    u_0(x, y) &= 0, \quad v_0(x, y) = 0. \nonumber
\end{align}
We note that this initial state is not in geostrophic equilibrium \eqref{eq:geo_eq}.  Figure \ref{fig:geostrophic_adjustment} displays the numerical solution for water height ($h$) (cell-averages) obtained using the Active Flux method at $t=4$ and $t=8$. It is observed that the method gives a good approximation of the outward-propagating discontinuities and preserves the overall structure of the solution.
\begin{figure}[htbp]
    \centering
    \begin{subfigure}{0.48\textwidth}
        \centering
        \includegraphics[width=\linewidth, height=4.1cm, keepaspectratio]{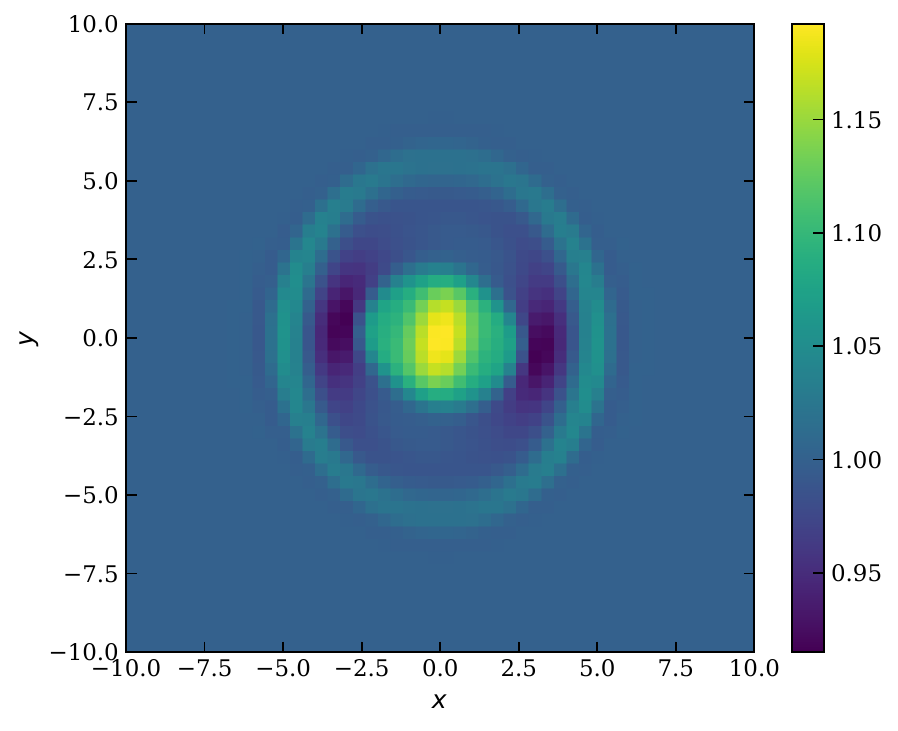} 
        \caption{Water height $h$ (cell-average) at $t=4$.}
    \end{subfigure}
    \hfill
    \begin{subfigure}{0.48\textwidth}
        \centering
        \includegraphics[width=\linewidth, height=4.1cm, keepaspectratio]{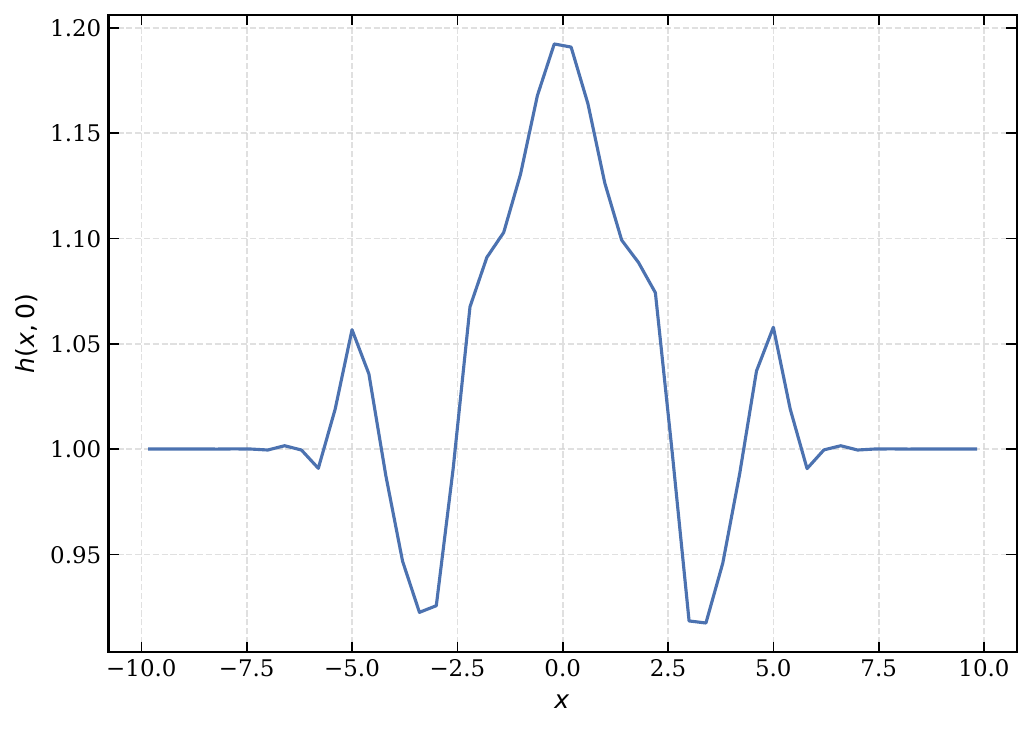}
        \caption{Cross-section along $y=0$ of $h$ at $t=4$.}
    \end{subfigure}

    \vspace{0.1cm} 

    \begin{subfigure}{0.48\textwidth}
        \centering
        \includegraphics[width=\linewidth, height=4.1cm, keepaspectratio]{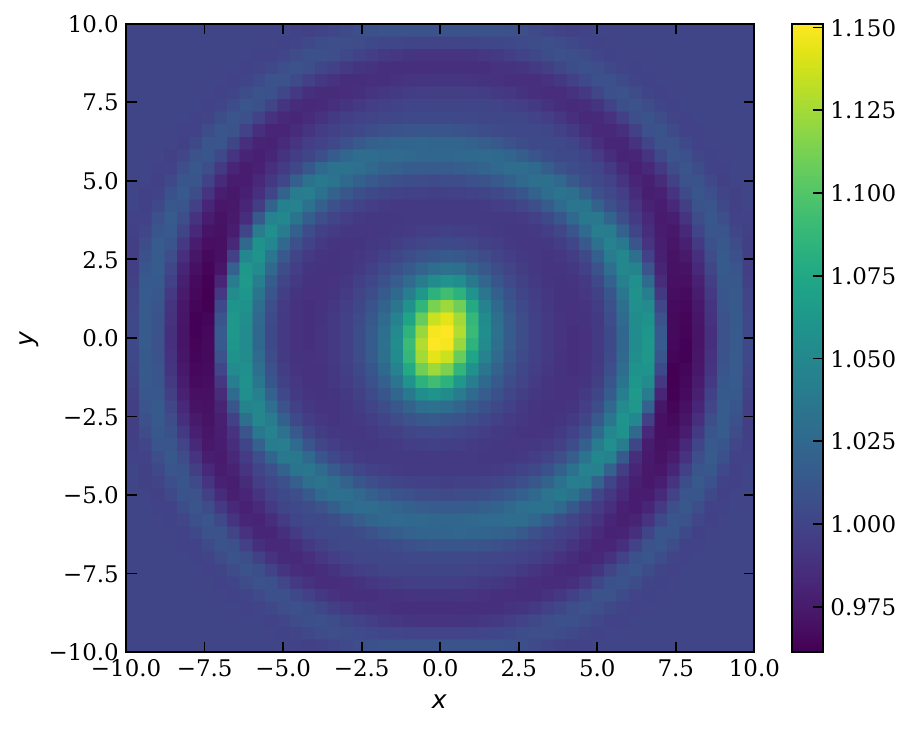} 
        \caption{Water height $h$ (cell-averages) at $t=8$.}
    \end{subfigure}
    \hfill
    \begin{subfigure}{0.48\textwidth}
        \centering
        \includegraphics[width=\linewidth, height=4.1cm, keepaspectratio]{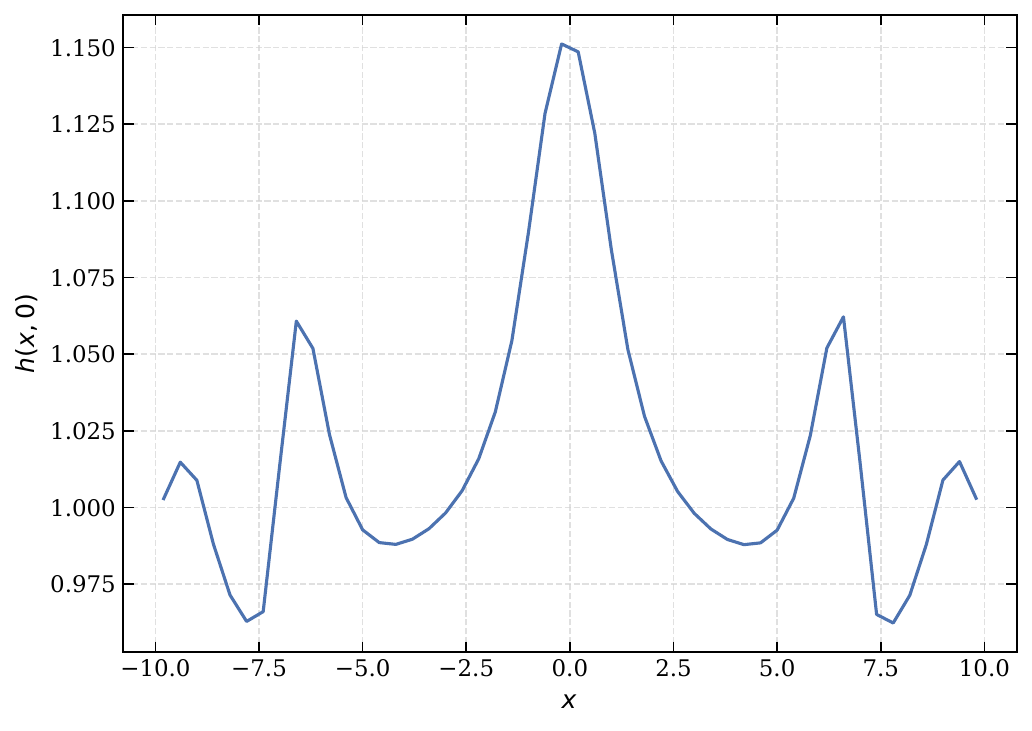}
        \caption{Cross-section along $y=0$ of $h$ at $t=8$.}
    \end{subfigure}
    
    \caption{Example \ref{ex:geo_adj}. Active Flux solution of the water height $h$ using a $50 \times 50$ grid. The top row shows the numerical solution at $t=4$, and the bottom row at $t=8$. The left panels display the cell-average solution, and the right panels display the corresponding 1D cross-section.}
    \label{fig:geostrophic_adjustment}
\end{figure}

\end{example}

\section{Conclusion}\label{sec:conclusion}
In this work, we investigated a semi-discrete Active Flux method for approximating linearized and nonlinear shallow water equations with Coriolis source terms. Our Fourier analysis established that for the linearized system, the method admits discrete stationary states corresponding to the analytical geostrophic equilibrium, proving it to be stationarity-preserving without requiring additional modifications. This highlights the intrinsic stationarity preservation properties of the Active Flux method, extending previous findings for homogeneous linear systems \cite{barsukow2023, Barsukow2025Fourier}. Numerical experiments, including simulation of a stationary vortex (Figure \ref{fig:combined_stat_vortex}) and that of a well-prepared initial datum (see Figure \ref{fig:af_wellprep}) validate the theoretical results for the method in the linear system case. Numerical investigation in the case of nonlinear shallow water equations also indicated good performance of the proposed method, especially in regimes near the geostrophic equilibrium. Future research will focus on deriving an energy stability property for the semi-discrete Active Flux method, as well as extending the theoretical well-balanced analysis to the method in the nonlinear system case.

\section*{Acknowledgements}
Wasilij Barsukow acknowledges support of the Agence Nationale de la Recherche (ANR) under grant ANR-25-CE40-3980 (project HoCo). Christian Klingenberg and Nikhil Manoj acknowledge funding by the German Science Foundation (DFG) KL 566/24-1 as part of a joint project with the Spanish Agencia Estatal de Investigación (AEI).

\bibliographystyle{plainurl}
\bibliography{reference}

\begin{appendices}
\section*{Appendix}
\section{Discrete differential operators in the Active Flux point value update}
\label{app:discrete_operators}
In this section, we elaborate on the formulas for the discrete differential operators used in the Active Flux point update \eqref{eq:af_ptvalue}. The discrete spatial derivatives in the $x$-direction evaluated at the node $P = (x_{i+1/2}, y_{j+1/2})$ are given by:
\begin{align}
    D_x^+ (q^{N}_{i, j}) &= \frac{1}{\Delta x} \left( 3 q^{N}_{i, j} + q^{N}_{i-1, j} - 4 {q}^{E_H}_{i,j} \right), \\
    D_x^- (q^{N}_{i, j}) &= \frac{1}{\Delta x} \left( -3 q^{N}_{i, j} - q^{N}_{i+1, j} + 4 {q}^{E_H}_{i+1,j} \right).
\end{align}
The discrete spatial derivatives in the $y$-direction evaluated at the node\\ $\displaystyle P = (x_{i+1/2}, y_{j+1/2})$  are given by:
\begin{align}
    D_y^+ (q^{{N}}_{i, j}) &= \frac{1}{\Delta y} \left( 3 q^{{N}}_{i, j} + q^{{N}}_{i, j-1} - 4 q^{E_V}_{i, j} \right), \\ 
    D_y^- (q^{{N}}_{i, j}) &= \frac{1}{\Delta y} \left( -3 q^{{N}}_{i, j} - q^{{N}}_{i, j+1} + 4 q^{E_V}_{i, j+1} \right).
\end{align}
The discrete spatial derivatives in the $x$-direction evaluated at the vertical edge midpoint $P = (x_{i+1/2}, y_j)$:
\begin{align}
    D_x^+(q^{E_V}_{i, j}) &= \frac{1}{\Delta x} \Bigg( 4 q^{E_V}_{i, j} + 2 q^{E_V}_{i-1, j} + q^{E_H}_{i, j} + q^{E_H}_{i, j-1} \nonumber \\
    &+ \frac{1}{4} \left( q^{N}_{i, j} + q^{N}_{i-1, j} + q^{N}_{i, j-1} + q^{N}_{i-1, j-1} \right) - 9 {q}^{A}_{i,j} \Bigg), \\
    D_x^- (q^{E_V}_{i, j}) &= \frac{1}{\Delta x} \Bigg( -4 q^{E_V}_{i, j} - 2 q^{E_V}_{i+1, j} - q^{E_H}_{i+1, j} - q^{E_H}_{i+1, j-1} \nonumber \\
    & \spc - \frac{1}{4} \left( q^{N}_{i, j} + q^{N}_{i+1, j} + q^{N}_{i, j-1} + q^{N}_{i+1, j-1} \right)   + 9 {q}^{A}_{i+1,j} \Bigg). \nonumber
\end{align}
The discrete spatial derivatives in the $y$-direction evaluated at the vertical edge midpoint $P = (x_{i+1/2}, y_j)$ :
\begin{align}
    &D_y^+ (q^{E_V}_{i, j}) = \frac{1}{\Delta y} \left( q^{{N}}_{i, j} - q^{{N}}_{i, j-1} \right), &
    D_y^- (q^{E_V}_{i, j}) = \frac{1}{\Delta y} \left( q^{{N}}_{i, j} - q^{{N}}_{i, j-1} \right).
\end{align}

The $x$ and $y$ spatial derivatives at the horizontal edge midpoint  $P = (x_{i}, y_{j+1/2})$ can be computed similarly.

\section{Update equations upon applying the Fourier transform}\label{appx:fourier_update}

The semi-discrete update formulas \eqref{eq:af_cellav_evo}-\eqref{eq:af_ptvalue} upon the the application of the Fourier transform are elaborated in this section. The update of the nodal point value can be written as
\begin{equation}
\begin{split}
    0 = \frac{\mathrm{d}}{\mathrm{d}t} \hat{q}^{N} &+ \left[ J_x^+ \frac{1}{\Delta x} \left(3 + \frac{1}{t_x}\right) + J_x^- \frac{1}{\Delta x} (-3 - t_x) + J_y^+ \frac{1}{\Delta y} \left(\frac{1}{t_y} + 3\right) \right. \\
    &\quad \left. + J_y^- \frac{1}{\Delta y} (-t_y - 3) \right] \hat{q}^{N} + \left[ \frac{-4}{\Delta x} (J_x^+ - J_x^- t_x) \right] \hat{q}^{\mathrm{E_H}} \\
    &+ \left[ \frac{-4}{\Delta y} (J_y^+ - J_y^- t_y) \right] \hat{q}^{\mathrm{E_V}} -{\begin{pmatrix}
    c\hat{v}^{N} \\ -c\hat{u}^{N} \\0 
    \end{pmatrix}} .
\end{split}
\end{equation}
The update equation for the point value at a vertical edge is
\begin{equation}
\begin{split}
    0& = \frac{\mathrm{d}}{\mathrm{d}t} \hat{q}^{\mathrm{E_V}} + \left[ -\frac{9}{\Delta x}(J_x^+ - J_x^- t_x) \right] \hat{q}^{\mathrm{A}} + \left[ \frac{1}{\Delta x} \left( J_x^+ \left(1 + \frac{1}{t_y}\right) + J_x^- \left(-\frac{t_x}{t_y} - t_x\right) \right) \right] \hat{q}^{\mathrm{E_H}} \\
    & \spc + \left[ \frac{2}{\Delta x} \left( J_x^+ \left(2 + \frac{1}{t_x}\right) + J_x^- (-2t_x - 2) \right) \right] \hat{q}^{\mathrm{E_V}} \\
    &\spc + \left[ J_x^+ \frac{1}{4\Delta x} \left(\frac{1}{t_x t_y} + \frac{1}{t_y} + 1 + \frac{1}{t_x}\right) + J_x^- \frac{1}{4\Delta x} \left(-\frac{1}{t_y} - \frac{t_x}{t_y} - t_x - 1\right) \right. \\
    &\spc \quad \left. + J_y^+ \frac{1}{\Delta y} \left(1 - \frac{1}{t_y}\right) + J_y^- \frac{1}{\Delta y} \left(1 - \frac{1}{t_y}\right) \right] \hat{q}^{N} -{\begin{pmatrix}
    c\hat{v}^{E_V} \\ -c\hat{u}^{E_V} \\0 
    \end{pmatrix}}
\end{split} 
\end{equation}
and for the update of the point value at a horizontal edge
\begin{equation}
\begin{split}
    0 &= \frac{\mathrm{d}}{\mathrm{d}t} \hat{q}^{\mathrm{E_H}}  + \left[ -\frac{9}{\Delta y}(J_y^+ - J_y^- t_y) \right] \hat{q}^{\mathrm{A}} + \left[ \frac{2}{\Delta y} \left( J_y^+ \left(\frac{1}{t_y} + 2\right) + J_y^- (-2 - t_y) \right) \right] \hat{q}^{\mathrm{E_H}} \\
    &\spc + \left[ \frac{1}{\Delta y} \left( J_y^+ \left(1 + \frac{1}{t_x}\right) + J_y^- \left(-t_y - \frac{t_y}{t_x}\right) \right) \right] \hat{q}^{\mathrm{E_V}} \\
    &\spc + \left[ J_x^+ \frac{1}{\Delta x} \left(1 - \frac{1}{t_x}\right) + J_x^- \frac{1}{\Delta x} \left(1 - \frac{1}{t_x}\right) + J_y^+ \frac{1}{4\Delta y} \left(\frac{1}{t_x t_y} + \frac{1}{t_y} + 1 + \frac{1}{t_x}\right) \right. \\
    &\spc \quad \left. + J_y^- \frac{1}{4\Delta y} \left(-\frac{1}{t_x} - 1 - t_y - \frac{t_y}{t_x}\right) \right] \hat{q}^{N} -{\begin{pmatrix}
    c\hat{v}^{E_H} \\ -c\hat{u}^{E_H} \\0 
    \end{pmatrix}}.
\end{split} 
\end{equation}

The update of the average is computed as follows
\begin{equation}
\begin{split}
    0 &= \frac{\mathrm{d}}{\mathrm{d}t} \hat{q}^{\mathrm{A}}  + \left[ \frac{2}{3\Delta y} J_y \left(1 - \frac{1}{t_y}\right) \right] \hat{q}^{\mathrm{E_H}} + \left[ \frac{2}{3\Delta x} J_x \left(1 - \frac{1}{t_x}\right) \right] \hat{q}^{\mathrm{E_V}} \\
    &\spc + \left[ \frac{1}{6\Delta x} J_x \left(1 - \frac{1}{t_x} + \frac{1}{t_y} - \frac{1}{t_x t_y}\right) + \frac{1}{6\Delta y} J_y \left(1 - \frac{1}{t_y} + \frac{1}{t_x} - \frac{1}{t_x t_y}\right) \right] \hat{q}^{N}  - {\begin{pmatrix}
    c\hat{v}^{A} \\ -c\hat{u}^{A} \\0 
    \end{pmatrix}}.
\end{split}
\end{equation}

\section{The evolution submatrices for semi-discrete Active Flux method}\label{appx:submatrices}

\subsection{Evolution submatrices for the Active Flux method}
Using the update formulas given in Section \ref{appx:fourier_update}, the submatrices in the evolution matrix $\mathcal{E}$  of semi-discrete Active Flux, given in \eqref{eq:evolutionmatrix_appx} can be derived as follows
\begin{equation*}
S_{cor} = 
\begin{bmatrix}
0 & c & 0 \\
-c & 0 & 0 \\
0 & 0 & 0
\end{bmatrix}, \quad \quad
\mathcal{E}_{AA} = -S_{cor} = 
\begin{bmatrix}
0 & -c & 0 \\
c & 0 & 0 \\
0 & 0 & 0
\end{bmatrix}
\end{equation*}

\begin{equation*}
\mathcal{E}_{E_H E_H} = 
\begin{bmatrix}
0 & -c & 0 \\
c & \frac{1}{\Delta y}(t_y^{-1} + 4 + t_y) & \frac{1}{\Delta y}(t_y^{-1} - t_y) \\
0 & \frac{1}{\Delta y}(t_y^{-1} - t_y) & \frac{1}{\Delta y}(t_y^{-1} + 4 + t_y)
\end{bmatrix}
\end{equation*}

\begin{equation*}
\mathcal{E}_{E_V E_V} = 
\begin{bmatrix}
\frac{1}{\Delta x}(t_x^{-1} + 4 + t_x) & -c & \frac{1}{\Delta x}(t_x^{-1} - t_x) \\
c & 0 & 0 \\
\frac{1}{\Delta x}(t_x^{-1} - t_x) & 0 & \frac{1}{\Delta x}(t_x^{-1} + 4 + t_x)
\end{bmatrix}
\end{equation*}

\begin{equation*}
\mathcal{E}_{NN} = 
\begin{bmatrix}
\frac{1}{2\Delta x}(t_x^{-1} + 6 + t_x) & -c & \frac{1}{2\Delta x}(t_x^{-1} - t_x) \\
c & \frac{1}{2\Delta y}(t_y^{-1} + 6 + t_y) & \frac{1}{2\Delta y}(t_y^{-1} - t_y) \\
\frac{1}{2\Delta x}(t_x^{-1} - t_x) & \frac{1}{2\Delta y}(t_y^{-1} - t_y) &  \left[ \frac{(t_x^{-1} + 6 + t_x)}{2\Delta x} + \frac{(t_y^{-1} + 6 + t_y)}{2\Delta y} \right]
\end{bmatrix}
\end{equation*}

\begin{equation*}
\mathcal{E}_{A E_H} = \frac{2}{3\Delta y}
\begin{bmatrix}
0 & 0 & 0 \\
0 & 0 & 1 - t_y^{-1} \\
0 & 1 - t_y^{-1} & 0
\end{bmatrix}, \quad \quad 
\mathcal{E}_{A E_V} = \frac{2}{3\Delta x}
\begin{bmatrix}
0 & 0 & 1 - t_x^{-1} \\
0 & 0 & 0 \\
1 - t_x^{-1} & 0 & 0
\end{bmatrix}
\end{equation*}

\begin{equation*}
\mathcal{E}_{AN} = 
\begin{bmatrix}
0 & 0 & \frac{(-1 + t_x)(1 + t_y)}{6\Delta x \, t_x t_y} \\
0 & 0 & \frac{(1 + t_x)(-1 + t_y)}{6\Delta y \, t_x t_y} \\
\frac{(-1 + t_x)(1 + t_y)}{6\Delta x \, t_x t_y} & \frac{(1 + t_x)(-1 + t_y)}{6\Delta y \, t_x t_y} & 0
\end{bmatrix}
\end{equation*}

\begin{equation*}
\mathcal{E}_{E_H A} = -\frac{9}{2\Delta y}
\begin{bmatrix}
0 & 0 & 0 \\
0 & 1 + t_y & 1 - t_y \\
0 & 1 - t_y & 1 + t_y
\end{bmatrix},
\end{equation*}

\begin{equation*}
\mathcal{E}_{E_H E_V} = \frac{1}{2\Delta y}
\begin{bmatrix}
0 & 0 & 0 \\
0 & \frac{(1 + t_x)(1 + t_y)}{t_x} & -\frac{(1 + t_x)(-1 + t_y)}{t_x} \\
0 & -\frac{(1 + t_x)(-1 + t_y)}{t_x} & \frac{(1 + t_x)(1 + t_y)}{t_x}
\end{bmatrix}
\end{equation*}

\begin{equation*}
\mathcal{E}_{E_H N} = 
\begin{bmatrix}
0 & 0 & \frac{1}{\Delta x}(1 - t_x^{-1}) \\
0 & \frac{(1 + t_x)(1 + t_y)^2}{8\Delta y \, t_x t_y} & -\frac{(1 + t_x)(-1 + t_y^2)}{8\Delta y \, t_x t_y} \\
\frac{1}{\Delta x}(1 - t_x^{-1}) & -\frac{(1 + t_x)(-1 + t_y^2)}{8\Delta y \, t_x t_y} & \frac{(1 + t_x)(1 + t_y)^2}{8\Delta y \, t_x t_y}
\end{bmatrix},
\end{equation*}

\begin{equation*}
\mathcal{E}_{E_V A} = -\frac{9}{2\Delta x}
\begin{bmatrix}
1 + t_x & 0 & 1 - t_x \\
0 & 0 & 0 \\
1 - t_x & 0 & 1 + t_x
\end{bmatrix} 
\end{equation*}

\begin{equation*}
\mathcal{E}_{E_V E_H} = \frac{1}{2\Delta x}
\begin{bmatrix}
\frac{(1 + t_x)(1 + t_y)}{t_y} & 0 & -\frac{(-1 + t_x)(1 + t_y)}{t_y} \\
0 & 0 & 0 \\
-\frac{(-1 + t_x)(1 + t_y)}{t_y} & 0 & \frac{(1 + t_x)(1 + t_y)}{t_y}
\end{bmatrix}
\end{equation*}

\begin{equation*}
\mathcal{E}_{E_V N} = 
\begin{bmatrix}
\frac{(1 + t_x)^2(1 + t_y)}{8\Delta x \, t_x t_y} & 0 & -\frac{(-1 + t_x^2)(1 + t_y)}{8\Delta x \, t_x t_y} \\
0 & 0 & \frac{1}{\Delta y}(1 - t_y^{-1}) \\
-\frac{(-1 + t_x^2)(1 + t_y)}{8\Delta x \, t_x t_y} & \frac{1}{\Delta y}(1 - t_y^{-1}) & \frac{(1 + t_x)^2(1 + t_y)}{8\Delta x \, t_x t_y}
\end{bmatrix}
\end{equation*}

\begin{equation*}
\mathcal{E}_{NA} = 
\begin{bmatrix}
0 & 0 & 0 \\
0 & 0 & 0 \\
0 & 0 & 0
\end{bmatrix},\quad \quad
\mathcal{E}_{N E_H} = -\frac{2}{\Delta x}
\begin{bmatrix}
1 + t_x & 0 & 1 - t_x \\
0 & 0 & 0 \\
1 - t_x & 0 & 1 + t_x
\end{bmatrix}, 
\end{equation*}
\begin{equation*}
\mathcal{E}_{N E_V} = -\frac{2}{\Delta y}
\begin{bmatrix}
0 & 0 & 0 \\
0 & 1 + t_y & 1 - t_y \\
0 & 1 - t_y & 1 + t_y
\end{bmatrix}
\end{equation*}
Only the submatrices $\displaystyle \mathcal{E}_{AA}, \mathcal{E}_{E_{H}E_{H}}, \mathcal{E}_{E_{V}E_{V}} \, \mbox{and} \, \mathcal{E}_{NN}$ along the diagonal of $\mathcal{E}$ have the influence of the source terms. This is owing to the fact that that in the semi-discrete Active Flux formulas, the discretization of the source terms in the cell-average evolution formula only involves the cell-averages of the conserved variables and the evolution formula for each point value involves only the respective point value of the conserved variable.

\subsection{Evolution submatrices for the central Active Flux method}
Following a similar procedure as in Section \ref{sec:fourier} for the central Active Flux method, we obtain a system of ordinary differential equations of the form \eqref{eq:evolutionmatrix} with an evolution  matrix $\mathcal{E}$ of the form \eqref{eq:evolutionmatrix_appx} with the submatrices given by\eqref{eq:evolutionmatrix_appx} with the submatrices given by
\begin{equation*}
\mathcal{E}_{AA} =  
\begin{bmatrix}
0 & -c & 0 \\
c & 0 & 0 \\
0 & 0 & 0
\end{bmatrix}, \quad
\mathcal{E}_{E_H E_H} = 
\begin{bmatrix}
0 & -c & 0 \\
c & 0 & \frac{1}{\Delta y}(t_y^{-1} - t_y) \\
0 & \frac{1}{\Delta y}(t_y^{-1} - t_y) & 0
\end{bmatrix},
\end{equation*}

\begin{equation*}
\mathcal{E}_{E_V E_V} = 
\begin{bmatrix}
0 & -c & \frac{1}{\Delta x}(t_x^{-1} - t_x) \\
c & 0 & 0 \\
\frac{1}{\Delta x}(t_x^{-1} - t_x) & 0 & 0
\end{bmatrix}, 
\end{equation*}
\begin{equation*}
\mathcal{E}_{NN} = 
\begin{bmatrix}
0 & -c & \frac{1}{2\Delta x}(t_x^{-1} - t_x) \\
c & 0 & \frac{1}{2\Delta y}(t_y^{-1} - t_y) \\
\frac{1}{2\Delta x}(t_x^{-1} - t_x) & \frac{1}{2\Delta y}(t_y^{-1} - t_y) & 0
\end{bmatrix},
\end{equation*}
\begin{equation*}
\mathcal{E}_{A E_H} = \frac{2}{3\Delta y}
\begin{bmatrix}
0 & 0 & 0 \\
0 & 0 & 1 - t_y^{-1} \\
0 & 1 - t_y^{-1} & 0
\end{bmatrix}, \quad \quad 
\mathcal{E}_{A E_V} = \frac{2}{3\Delta x}
\begin{bmatrix}
0 & 0 & 1 - t_x^{-1} \\
0 & 0 & 0 \\
1 - t_x^{-1} & 0 & 0
\end{bmatrix},
\end{equation*}
\begin{equation*}
\mathcal{E}_{AN} = 
\begin{bmatrix}
0 & 0 & \frac{(-1 + t_x)(1 + t_y)}{6\Delta x \, t_x t_y} \\
0 & 0 & \frac{(1 + t_x)(-1 + t_y)}{6\Delta y \, t_x t_y} \\
\frac{(-1 + t_x)(1 + t_y)}{6\Delta x \, t_x t_y} & \frac{(1 + t_x)(-1 + t_y)}{6\Delta y \, t_x t_y} & 0
\end{bmatrix},
\end{equation*}

\begin{equation*}
\mathcal{E}_{E_H A} = -\frac{9}{2\Delta y}
\begin{bmatrix}
0 & 0 & 0 \\
0 & 0 & 1 - t_y \\
0 & 1 - t_y & 0
\end{bmatrix},
\end{equation*}

\begin{equation*}
\mathcal{E}_{E_H E_V} = \frac{1}{2\Delta y}
\begin{bmatrix}
0 & 0 & 0 \\
0 & 0 & -\frac{(1 + t_x)(-1 + t_y)}{t_x} \\
0 & -\frac{(1 + t_x)(-1 + t_y)}{t_x} & 0
\end{bmatrix},
\end{equation*}

\begin{equation*}
\mathcal{E}_{E_H N} = 
\begin{bmatrix}
0 & 0 & \textcolor{black}{\frac{1}{\Delta x}(1 - t_x^{-1})} \\
0 & 0 & -\frac{(1 + t_x)(-1 + t_y^2)}{8\Delta y \, t_x t_y} \\
\frac{1}{\Delta x}(1 - t_x^{-1}) & -\frac{(1 + t_x)(-1 + t_y^2)}{8\Delta y \, t_x t_y} & 0
\end{bmatrix},
\end{equation*}

\begin{equation*}
\mathcal{E}_{E_V A} = -\frac{9}{2\Delta x}
\begin{bmatrix}
0 & 0 & 1 - t_x \\
0 & 0 & 0 \\
1 - t_x & 0 & 0
\end{bmatrix},
\end{equation*}

\begin{equation*}
\mathcal{E}_{E_V E_H} = \frac{1}{2\Delta x}
\begin{bmatrix}
0 & 0 & -\frac{(-1 + t_x)(1 + t_y)}{t_y} \\
0 & 0 & 0 \\
-\frac{(-1 + t_x)(1 + t_y)}{t_y} & 0 & 0
\end{bmatrix},
\end{equation*}

\begin{equation*}
\mathcal{E}_{E_V N} = 
\begin{bmatrix}
0 & 0 & -\frac{(-1 + t_x^2)(1 + t_y)}{8\Delta x \, t_x t_y} \\
0 & 0 & \textcolor{black}{\frac{1}{\Delta y}(1 - t_y^{-1})} \\
-\frac{(-1 + t_x^2)(1 + t_y)}{8\Delta x \, t_x t_y} & \frac{1}{\Delta y}(1 - t_y^{-1}) & 0
\end{bmatrix},
\quad \mathcal{E}_{NA} = 
\begin{bmatrix}
0 & 0 & 0 \\
0 & 0 & 0 \\
0 & 0 & 0
\end{bmatrix},
\end{equation*}

\begin{equation*}
\mathcal{E}_{N E_H} = -\frac{2}{\Delta x}
\begin{bmatrix}
0 & 0 & 1 - t_x \\
0 & 0 & 0 \\
1 - t_x & 0 & 0
\end{bmatrix}, \quad \quad
\mathcal{E}_{N E_V} = -\frac{2}{\Delta y}
\begin{bmatrix}
0 & 0 & 0 \\
0 & 0 & 1 - t_y \\
0 & 1 - t_y & 0
\end{bmatrix}.
\end{equation*}

\end{appendices}

\end{document}